\documentclass [a4paper,10pt]{amsart}
\usepackage{longtable}
\usepackage{hyperref}
\usepackage[T1]{fontenc}
\usepackage[utf8]{inputenc}
\usepackage[english]{babel}
\usepackage{textcomp}
\usepackage{dsfont}
\usepackage{latexsym}
\usepackage{amssymb}
\usepackage{amsthm}
\usepackage{amsmath}
\DeclareMathAlphabet{\mathpzc}{OT1}{pzc}{m}{en}
\usepackage{yfonts}
\usepackage{xfrac}
\usepackage{newlfont}
\usepackage{graphicx}
\usepackage{mathtools}
\usepackage{comment}
\usepackage{indentfirst}
\usepackage{braket}
\usepackage{mathrsfs}
\usepackage{xcolor}
\usepackage{fixmath}

\usepackage{etoolbox}
\apptocmd{\lim}{\limits}{}{}
\apptocmd{\sup}{\limits}{}{}
\apptocmd{\inf}{\limits}{}{}
\apptocmd{\liminf}{\limits}{}{}
\apptocmd{\limsup}{\limits}{}{}
\usepackage{scalerel}[2014/03/10]
\usepackage[usestackEOL]{stackengine}
\newcommand{\dashint}{\,\ThisStyle{\ensurestackMath{%
			\stackinset{c}{.2\LMpt}{c}{.5\LMpt}{\SavedStyle-}{\SavedStyle\phantom{\int}}}%
		\setbox0=\hbox{$\SavedStyle\int\,$}\kern-\wd0}\int}

\DeclareMathOperator{\card}{Card}

\DeclareMathOperator{\pr}{pr}

\DeclareMathOperator{\tr}{Tr}

\DeclareMathOperator{\Hol}{Hol}

\renewcommand{\Re}{\mathrm{Re}\,}
\renewcommand{\Im}{\mathrm{Im}\,}

\newcommand{\ee}{\mathrm{e}}
\newcommand{\Aff}{\mathrm{Aff}}

\newcommand{\vect}[1]{\mathbf{{#1}}}
\newcommand{\dd}{\mathrm{d}}

\DeclarePairedDelimiter{\abs}{\lvert}{\rvert}

\DeclarePairedDelimiter{\norm}{\lVert}{\rVert}

\let\originalleft\left
\let\originalright\right
\renewcommand{\left}{\mathopen{}\mathclose\bgroup\originalleft}
\renewcommand{\right}{\aftergroup\egroup\originalright}

\newcommand{\N}{\mathds{N}}

\newcommand{\C}{\mathds{C}}

\newcommand{\R}{\mathds{R}}

\newcommand{\Ms}{\mathscr{M}}

\newcommand{\Af}{\mathfrak{A}}
\newcommand{\Bf}{\mathfrak{B}}

\newcommand{\Kf}{\mathfrak{K}}
\newcommand{\Lf}{\mathfrak{L}}

\newcommand{\Pf}{\mathfrak{P}}

\newcommand{\Ac}{\mathcal{A}}
\newcommand{\Bc}{\mathcal{B}}

\newcommand{\Ec}{\mathcal{E}}
\newcommand{\Fc}{\mathcal{F}}

\newcommand{\Hc}{\mathcal{H}}

\newcommand{\Lc}{\mathcal{L}}

\newcommand{\Nc}{\mathcal{N}}

\newcommand{\Sc}{\mathcal{S}}

\newcommand{\Uc}{\mathcal{U}}

\newcommand{\meg}{\leqslant}
\newcommand{\Meg}{\geqslant}
\newcommand{\eps}{\varepsilon}
\renewcommand{\phi}{\varphi}

\newcommand{\leftexp}[2]{{\vphantom{#2}}^{#1}{#2}} 
\newcommand{\trasp}{\leftexp{t}}
\newcommand{\Lin}{\mathscr{L}}

\begin{document}

\theoremstyle{definition}
\newtheorem{deff}{Definition}[section]

\newtheorem{oss}[deff]{Remark}

\newtheorem{ass}[deff]{Assumptions}

\newtheorem{nott}[deff]{Notation}

\theoremstyle{plain}
\newtheorem{teo}[deff]{Theorem}

\newtheorem{lem}[deff]{Lemma}

\newtheorem{prop}[deff]{Proposition}

\newtheorem{cor}[deff]{Corollary}

\title[Bergman Spaces on Homogeneous Siegel Domains]{On the Theory of Bergman Spaces on Homogeneous Siegel Domains}

\author[M. Calzi, M. M. Peloso]{Mattia Calzi, 
	Marco M. Peloso}

\date{}

\address{Dipartimento di Matematica, Universit\`a degli Studi di
	Milano, Via C. Saldini 50, 20133 Milano, Italy}
\email{{\tt mattia.calzi@unimi.it}}
\email{{\tt marco.peloso@unimi.it}}

\thanks{{\em Math Subject Classification (2020): 32A15, 32A37, 32A50, 46E22}}
\thanks{The authors are members of the 	Gruppo Nazionale per l'Analisi
  Matematica, la Probabilit\`a e le	loro Applicazioni (GNAMPA) of
  the Istituto Nazionale di Alta Matematica (INdAM). The authors were partially funded by the INdAM-GNAMPA Project CUP\_E55F22000270001.
} 

\begin{abstract}
We consider mixed normed Bergman spaces on homogeneous
Siegel domains.  In the literature, two different approaches have been
considered and several results seem difficult to be compared.  In this
paper we compare the results available in the literature and complete
the existing ones in one of the two settings. 
The results we present are:
natural inclusions, density, completeness, reproducing properties,
sampling, atomic decomposition, duality, continuity of Bergman
projectors, boundary values, transference. 
\end{abstract}
\maketitle

\section{Introduction}

This paper deals with some spaces of holomorphic functions on
a homogeneous Siegel domain.  In order to illustrate the
kind of spaces and problems we are going to consider, we begin with
the simplest case.

Let  
$\C_+ \coloneqq \Set{z\in\C\colon \Im z>0}$ be the upper half-plane.
We can think of $\C_+$ as $\R+i(0,\infty)$, where $(0,\infty)$ is the
unique open (convex) cone in $\R$.  In several variables, the upper
half-plane can  be generalized to tube domains over convex cones.
Let $\Omega$ be a convex open cone in $\R^m$.  Then, the domain
$D=\R^m +i\Omega$ in $\C^m$ is called the tube domain over the cone
$\Omega$.  If the group of linear
transfomations of $\R^m$ that preserve $\Omega$ acts transitively on
$\Omega$ itself, then $\Omega$ is a {\em homogeneous} cone and the
domain becomes itself homogeneous, that is, the group of biholomorphic
self-maps of $D$ (the {\em automorphisms} of $D$) acts transitively on
$D$.

Another classical domain in several variables that extends
the definition and some of the main features of $\C_+$ is the
so-called Siegel upper half-space.  Consider again the cone
$(0,\infty)\subseteq\R$ and the hermitian quadratic map on $\C^n$
$\zeta\mapsto \zeta\cdot\overline\zeta=|\zeta|^2$. Then, the Siegel
upper half-space is the domain $\Uc$ in $\C^n\times\C$
\[
\Uc \coloneqq  \Set{(z,\zeta)\colon \, \Im z -\abs{\zeta}^2\in (0,\infty) }.
\]

The {\em homogeneous Siegel domains} are then introduced as
follows -- we refer to Section~\ref{sec:2} for complete definitions.
Let 
a homogeneous cone $\Omega\subset \R^m$ and a suitable hermitian quadratic map
$\Phi\colon \C^n\to \C^m$ be given.  Then, the homogeneous Siegel
domain $D\subseteq\C^n\times \C^m$ is 
\[
D=\Set{(\zeta,z)\in \C^n\times \C^m\colon \Im z-\Phi(\zeta)\in \Omega}.
\] 
(again, cf.~Section~\ref{sec:2} for definitions).
Notice that if $n=0$, then $D$ is the tube domain over the given cone
$\Omega$. 

\smallskip

On a  homogeneous Siegel domain $D$ as above,
various (mixed norm) weighted Bergman spaces have been considered in the literature. On the one hand, in~\cite{RicciTaibleson} (for the upper half-plane) and~\cite{CalziPeloso} (for the general case), the following mixed norm weighted Bergman spaces are considered:
\[
A^{p,q}_{\vect s}\coloneqq \Set{f\in \Hol(D)\colon \norm{h\mapsto \Delta^{\vect s}_\Omega(h) \norm{f_h}_{L^p(\Nc)}}_{L^q(\nu_\Omega)}<\infty  },
\]
where $\Delta^{\vect s}_\Omega$ are `generalized power functions' on $\Omega$ ($\vect s\in \R^r$), $\nu_\Omega$ is `the' invariant measure on $\Omega$, $\Nc=\C^n\times \R^m$ and $f_h\colon (\zeta,x)\mapsto f(\zeta,x+i\Phi(\zeta)+i h)$. On the other hand, e.g.~in~\cite{Bekolle,Bonami,Bekolleetal,BekolleBonamiGarrigosRicci,Debertol,Nana,BekolleGonessaNana,Bekolleetal2,BekolleGonessaNana2}, the following mixed norm weighted Bergman spaces are considered:
\[
\Af^{p,q}_{\vect s}\coloneqq \Set{f\in \Hol(D)\colon \norm{h\mapsto \norm{f_h}_{L^{p,q}(\R^m,\C^n)}}_{L^q(\Delta_\Omega^{\vect s+\vect b/2}\cdot\nu_\Omega)}<\infty  },
\]
where $\vect b$ is a suitable element of $\R^r$ and $L^{p,q}(\R^m,\C^n)=\Set{f\colon \norm{\zeta \mapsto \norm{f(\zeta,\,\cdot\,)}_{L^p(\R^m)}  }_{L^q(\C^n)}<\infty}$.

Two parallel theories then arise, and different conventions have been adopted. For example, the definition of the spaces $\Af^{p,q}_{\vect s}$ suggests a natural comparison between the spaces $\Af^{p,q}_{\vect s}$ for a fixed $\vect s$, which in turn highlights the role played by `the' Bergman projector $\Pf_{\vect s}$, namely the Bergman projector of the corresponding space $\Af^{2,2}_{\vect s}$.  On the other hand, the comparison of the spaces $A^{p,q}_{\vect s}$ for fixed $\vect s$ appears to be less natural, so that more general Bergman projectors are naturally investigated. 
Besides that, in the study of various properties of the spaces
$\Af^{p,q}_{\vect s}$ (such as, for instance, the continuity of
$\Pf_{\vect s}$ on the space $\Lf^{p,q}_{\vect s}$, which is defined
as $\Af^{p,q}_{\vect s}$ replacing holomorphic functions with
equivalence classes of measurable functions) greater attention is
posed on $p$ and $q$, rather than $\vect s$, whereas in the study of
the spaces $A^{p,q}_{\vect s}$ greater attention is posed on $\vect
s$, rather than $p$ or $q$, so that even when describing the same
phenomena, the two parallel theories may appear quite different from
on another, and hard to compare. 

\smallskip

This is the main reason which motivated us to write this work, where we describe the various analogous results for the spaces $A^{p,q}_{\vect s}$ and $\Af^{p,q}_{\vect s}$ for a direct comparison. In addition, we deepen the study of the spaces $\Af^{p,q}_{\vect s}$ proving those results which, to the best of our knowledge, do not appear in the literature in the present generality. In order to tackle these issues, we are naturally led to introduce a new family of spaces $\Ac^{p,q}_{\vect s}$, which is defined in the spirit of the spaces $A^{p,q}_{\vect s}$ (and has, therefore, some technical advantages), but allows to treat   also the spaces $\Af^{p,q}_{\vect s}$ by means of straightforward substitutions.

\smallskip

Function theory and analysis of function spaces on homogeneous Siegel
domains is a classical area of research in which complex analysis,
harmonic analysis, geometry of convex cones, representation theory all
play a fundamental role, see e.g.~\cite{Gindikin,Murakami,OV,PS,RV} and
also~\cite{FarautKoranyi}.  In more recent times, it was shown
in~\cite{Bekolleetal3} and then generalized in~\cite{Bekolleetal2}, that in
order to prove the $L^p$-boundedness of the Bergman
projector on some homogeneous Siegel domain of tube-type (see
Subsection~\ref{tube-domain:subsec}) 
it was necessary to exploit the cancellations of the kernel,
phenomenon that had never observed before, and in fact these examples
remain the only instances of such behavior.  We mention in passing that, in order to
exploit the cancellation of the Bergman kernel, the mixed-norm
spaces were considered, hence showing their naturality in this context.  It is also worth mentioning that the question of the
$L^p$-boundeness of the Bergman projector on homoegeneous Siegel
domains is tightly
connected to 
 the sharp $\ell^2$-decoupling inequality of Bourgain and
 Demeter~\cite{BourgainDemeter}, see~\cite{BekolleGonessaNana}.  We
 refer also to the Introduction in~\cite{Paralipomena} for a through
 discussion of this and related questions, and
 to~\cite{Nana,BekolleIshiNana,BekolleGonessaNana,BekolleGonessaNana2,CalziPeloso,Toeplitz-Cesaro,Carleson-reverseCarleson,Paralipomena}
 for some very recent works on the subject of the current paper.

 \smallskip

The  paper is organized as follows.
In Section~\ref{sec:2}, we recall some basic facts and introduce some notation to deal with homogeneous Siegel domains. In order to facilitate readers who are accostumed to different conventions, we introduce our notation axiomatically, allowing the reader to identify the (hopefully minimal) modifications needed. 
In Section~\ref{sec:3}, we briefly list several resuts for the spaces $A^{p,q}_{\vect s}$ and $\Af^{p,q}_{\vect s}$, deferring the proofs to Section~\ref{sec:proofs}, where a more detailed treatment of the spaces $\Ac^{p,q}_{\vect s}$ is presented. We nontheless present only those proofs which require substantial modifications to the arguments presented for the corresponding proofs in~\cite{CalziPeloso,Paralipomena}. 
The topics we shall present include: natural inclusions, density, completeness, reproducing properties, sampling, atomic decomposition, duality, continuity of Bergman projectors, boundary values, transference.

\section{Homogeneous Siegel Domains}\label{sec:2}

We denote by $E$ a complex Hilbert space of dimension $n$, by $F$ a real Hilbert space of dimension $m$, by $\Omega$ an open convex cone in $F$ not containing affine lines, and by $\Phi\colon E\times E\to F_\C$ a non-degenerate hermitian map such that $\Phi(\zeta)\coloneqq \Phi(\zeta,\zeta)\in \overline\Omega$ for every $\zeta\in E$.
We define
\[
\rho\colon E\times F_\C\ni (\zeta,z)\mapsto \Im z-\Phi(\zeta)\in F,
\] 
and denote by 
\[
D=\Set{(\zeta,z)\in E\times F_\C\colon \Im z-\Phi(\zeta)\in \Omega}=\rho^{-1}(\Omega)
\]
the Siegel domain associated with $\Omega$ and $\Phi$. We shall assume that $D$ is homogeneous, that is, that the group of its biholomorphisms acts transitively on it. It is known (cf., e.g.,~\cite[Proposition 1]{Besov}) that $D$ is homogeneous if and only if there is a triangular\footnote{This means that all the eigenvalues of every element of $T_+$ are real. Equivalently, there is a basis of $F$ with respect to which every element of $T_+$ is represented by an upper triangular matrix, cf.~\cite{Vinberg2}.} subgroup $T_+$ of $GL(F)$ which  acts simply transitively on $\Omega$, and for every $t\in T_+$ there is $g\in GL(E)$ such that $t\circ \Phi=\Phi\circ(g\times g)$. 
In this case, any other triangular subgroup of $GL(F)$ with the same properties is conjugated to $T_+$ by an element of $GL(F)$ which preserves $\Omega$. In addition, $T_+$ acts simply transitively on the right on $\Omega'$, by transposition (cf.~\cite[Theorem 1]{Vinberg}). We shall denote this latter action by $\lambda\cdot t$, for $\lambda\in \Omega'$ and $t\in T_+$; we shall consequently write $t\cdot h$ instead of $t h$ for $t\in T_+$ and $h\in \Omega$. We shall still denote by $t\,\cdot $ and $\cdot\, t$ the actions of $t$ on $F_\C$ and $F'_\C$, respectively, for every $t\in T_+$.

\subsection{Analysis on $\Omega$}

It is possible to describe the structure of $T_+$ and of its action on $\Omega$ using the theory of $T$-algebras, cf.~\cite{Vinberg}, or the theory of (normal) $j$-algebras, cf.~\cite{PS,RV}. In order to keep the exposition as simple as possible, we shall avoid a thorough description of the structure of $T_+$ and proceed axiomatically. We refer the reader to~\cite{CalziPeloso} for a more detailed treatment of the following considerations. 
We first observe that there are $r\in \N$ (called the rank of $\Omega$) and a surjective homomorphism of Lie groups
\[
\Delta\colon T_+\to (\R_+^*)^r,
\]
with kernel $[T_+,T_+]$,
such that, if we fix base-points $e_\Omega\in \Omega$ and $e_{\Omega'}\in \Omega'$ and define
\begin{equation}\label{eq:1}
\Delta_\Omega^{\vect s}(t\cdot e_\Omega)= \Delta^{\vect s}_{\Omega'}(e_{\Omega'}\cdot t)=\Delta^{\vect s}(t)=\prod_{j=1}^r \Delta_j(t)^{s_j}
\end{equation}
for every $\vect s\in \C^r$ and for every $t\in T_+$, then $\Delta_\Omega^{\vect s}$ (and $\Delta_{\Omega'}^{\vect s}$) is bounded on bounded subsets if and only if $\Re\vect s\in \R_+^r$ (cf.~\cite[Lemma 2.34]{CalziPeloso}). We shall further require that $\Delta(a\,\cdot\,)=(a,\dots,a)$ for every $a>0$, where $a\,\cdot\,$ denotes the homothety of ratio $a$ (which necessarily belongs to $T_+$).
We remark explicitly that  these conditions determine $\Delta$ up to a permutation  of the coordinates (in $(\R_+^*)^r$).\footnote{To see this fact, observe that, if $\Delta'\colon T_+\to (\R_+^*)^r$ is another homomorphism with the same properties, then there is $A\in GL(\R^n)$ such that $\log\Delta'=A\log \Delta$.  In addition, given $\vect s\in \R^r$, both $\sum_j s_j\log \Delta'_j=\sum_j (\trasp A\vect s)_j \log \Delta_j$ and $\sum_j s_j \log \Delta_j$ induce functions which are bounded above on the bounded subsets of $\Omega$ if and only if $\vect s \in \R_+^*$, so that $\trasp A\R_+^*=\R_+^*$ and therefore $A$ must be the composition of a permutation of the coordinates and a diagonal dilation $(x_1,\dots,x_r)\mapsto (\lambda_1 x_1,\dots, \lambda_r x_r)$, $\lambda_1,\dots, \lambda_r>0$. Since $\Delta(a\,\cdot\,)=\Delta'(a\,\cdot\,)=(a,\dots,a)$ for every $a>0$, we then see that $A$ must induce the identity on the line $\R\vect 1_r$, so that it must be a permutation of the coordinates. } Consequently, we may apply the results of~\cite[Chapter 2]{CalziPeloso} without (essential) changes, even if a different choice of $T_+$ and $\Delta$ is made. 
Notice that $\Delta_\Omega^{\vect s}$ and $\Delta_{\Omega'}^{\vect s}$ extend to holomorphic functions on $\Omega+i F$ and $\Omega'+i F'$, respectively, for every $\vect s\in \C^r$ (cf.~\cite[Corollary 2.25]{CalziPeloso}). 

When $\Omega$ is symmetric, that is, self-dual with respect to the scalar product of $F$, then the functions $\Delta_{\vect s}$ considered in~\cite{FarautKoranyi} coincide with the functions $\Delta^{\vect s}_{\Omega}$ defined in~\eqref{eq:1} for an appropriate choice of $\Delta$ (cf.~\cite[Chapter VI, \S\ 3]{FarautKoranyi}); in particular, the `determinant' polynomial concides with $\Delta^{\vect 1_r}_\Omega$. Generally speaking, the works which deal with the case in which $\Omega$ is symmetric generally adhere to the conventions of~\cite{FarautKoranyi}, possibly with slightly different notation, whereas the works which deal with general homogeneous cones generally adhere to the conventions described above, possibly with different notation (for example, $\Delta_\Omega^{\vect s}=Q^{\vect s}$ and $\Delta_{\Omega'}^{\vect s}=(Q^*)^{\vect s}$ in the notation of~\cite{Nana,NanaTrojan,BekolleIshiNana,BekolleGonessaNana})

To simplify the notation, we state the following definition.

\begin{deff}
	We define two order relations on $\R^r$. On the one hand, we write $\vect s \meg \vect{s'}$  to mean $s_j\meg s'_j$ for every $j=1,\dots,r$ (equivalently, $\vect {s'}-\vect s\in\R_+^r$). On the other hand, we write $\vect s \preceq \vect{s'}$  to mean $\vect s=\vect{s'}$ or $s_j< s'_j$ for every $j=1,\dots,r$. 
\end{deff}

Thus, $\vect s\prec \vect{s'}$ (that is, $\vect s\preceq \vect s'$ and $\vect s\neq \vect s'$) if and only if $\vect{s'}-\vect s\in (\R_+^*)^r$, that is, $s_j<s'_j$ for every $j=1,\dots, r$.

\begin{deff} We denote by $\Hc^k$ the $k$-dimensional Hausdorff measure.
	There are $\vect d\prec \vect 0$ and $\vect b\meg \vect0$ such that
	\begin{equation}\label{eq:2}
	\nu_\Omega\coloneqq\Delta_\Omega^{\vect d}\cdot \Hc^m, \qquad \nu_{\Omega'}\coloneqq\Delta_{\Omega'}^{\vect d}\cdot \Hc^m, \qquad \text{and} \qquad \nu_D\coloneqq(\Delta_{\Omega}^{\vect b+2 \vect d}\circ \rho)\cdot \Hc^{2n+2m}
	\end{equation}
	are the unique measures on $\Omega$, $\Omega'$, and $D$ (up to a multiplicative constant) which are invariant under all linear automorphisms of $\Omega$ and $\Omega'$, and all biholomorphisms of $D$, respectively (cf.~\cite[Propositions 2.19 and 2.44]{CalziPeloso}, and~\cite[Proposition I.3.1]{FarautKoranyi}). 
\end{deff}	
	
\begin{oss}
	Notice that $\Delta^{-\vect b}(t)= \abs{\det_\C g}^2$ for every $t\in T_+$ and for every $g\in GL(E)$ such that $t\cdot \Phi=\Phi\circ (g\times g)$. Further, $\Delta^{-\vect d}(a\,\cdot\,)=a^m$, and $\Delta^{-\vect b}(a\,\cdot\,)=a^n$ for every $a>0$.
	
	We observe explicitly that $\vect d=d$ and $\vect b=-q$ in the notation of~\cite{Gindikin,Bekolle,BekolleKagou}, whereas $\vect d=-\tau$ and $\vect b=-b$ in the notation of~\cite{NanaTrojan,Nana,BekolleIshiNana,BekolleGonessaNana}. In particular, there is no general agreement on the sign of $\vect d$.
\end{oss}

\begin{deff}
	There are $\vect m,\vect{m'}\Meg \vect 0$ such that  $\Delta_\Omega^{\vect s}\cdot \nu_\Omega$ and $\Delta_{\Omega'}^{\vect s}\cdot \nu_{\Omega'}$ induce Radon measures on $F$ and $F'$, respectively, if and only if $\Re \vect s\succ\frac 1 2 \vect m$ and $\Re \vect s\succ \frac 1 2 \vect{m'}$, respectively (cf.~\cite[Proposition 2.19]{CalziPeloso}).  
\end{deff}

\begin{oss}
	Notice that  $\vect d=-(\vect 1_r+\frac 1 2 \vect m+\frac 1 2 \vect m')$ (cf.~\cite[Definition 2.8]{CalziPeloso} and the preceding remarks).
	We observe explicitly that $\vect m=(m_1,\dots, m_r)$ and $\vect m'=(n_1,\dots,n_r)$ in the notation of~\cite{BekolleKagou,NanaTrojan,Nana,BekolleIshiNana,BekolleGonessaNana,BekolleGonessaNana2}.
\end{oss}

\begin{deff}
	For every $\vect s,\vect s'\in \C^r $ such that $\Re\vect s\succ \frac 1 2 \vect m$ and $\Re\vect s'\succ\frac 1 2 \vect m'$, we define $\Gamma_\Omega(\vect s)$ and $\Gamma_{\Omega'}(\vect s')$ so that 
	\[
	\Lc(\Delta_\Omega^{\vect s}\cdot \nu_\Omega)=\Gamma_\Omega(\vect s)\Delta^{-\vect s}_{\Omega'} \qquad  \text{and}\qquad \Lc(\Delta_{\Omega'}^{\vect s'}\cdot \nu_{\Omega'})=\Gamma_{\Omega'}(\vect s')\Delta^{-\vect s'}_\Omega,
	\]
	respectively, where $\Lc$ denotes the Laplace transform.
\end{deff}

\begin{oss}
	Notice that $\Gamma_\Omega(\vect s)=\Lc(\Delta_\Omega^{\vect s}\cdot \nu_\Omega)(e_{\Omega'})=c \prod_{j=1}^r \Gamma\big(s_j-\frac 1 2 m_j\big)$ and $\Gamma_{\Omega'}(\vect s')=\Lc(\Delta_{\Omega'}^{\vect s'}\cdot \nu_{\Omega'})(e_\Omega)=c' \prod_{j=1}^r \Gamma\big(s'_j-\frac 1 2 m'_j\big)$ for some constants $c,c'>0$ which depend on the choice of $e_\Omega$ and $e_{\Omega'}$.
\end{oss}

\begin{deff}
	There are two uniquely determined holomorphic families $(I^{\vect s}_\Omega)_{\vect s\in \C^r}$ and $(I^{\vect s}_{\Omega'})_{\vect s\in \C^r}$ of tempered distributions on $F$ and $F'$, respectively, such that $\Lc I^{\vect s}_\Omega=\Delta^{-\vect s}_{\Omega'}$ and $\Lc I^{\vect s}_{\Omega'}=\Delta^{-\vect s}_\Omega$ (cf.~\cite[Lemma 2.26 and Proposition 2.28]{CalziPeloso}). 
\end{deff}

\begin{oss}
	Notice that $I^{\vect s}_\Omega=\frac{1}{\Gamma_\Omega(\vect s)}\Delta_\Omega^{\vect s}\cdot \nu_\Omega$ and $I^{\vect s}_{\Omega'}=\frac{1}{\Gamma_{\Omega'}(\vect s)}\Delta^{\vect s}_{\Omega'}\cdot \nu_{\Omega'}$ when $\Re \vect s\succ\frac 1 2 \vect m$ and $\Re \vect s\succ \frac  1 2 \vect{m'}$, respectively. In addition, $I^{\vect s}_\Omega$ and $I^{\vect s}_{\Omega'}$ are supported in $\overline \Omega$ and $\overline{\Omega'}$, respectively, for every $\vect s\in \C^r$ (cf.~\cite[Proposition 2.28]{CalziPeloso}).
\end{oss}

\begin{deff}
	We denote by $\N_\Omega$ and $\N_{\Omega'}$ the sets of $\vect s\in \R^r$ such that $\Delta_\Omega^{\vect s}$ and $\Delta_{\Omega}^{\vect s}$ extend to polynomials on $F$ and $F'$, respectively.
\end{deff}

\begin{oss}
	Notice that $I^{\vect s}_{\Omega}$ and $I^{\vect s}_{\Omega'}$ are supported in $\Set{0}$ if and only if $\vect s\in -\N_{\Omega'}$ and $\vect s\in -\N_\Omega$, respectively.
	Then, $\Phi_*(\Hc^{2n})=c I^{-\vect b}_{\Omega}$ for a suitable constant $c>0$ which depends on the choice of $e_{\Omega'}$ (cf.~\cite[Proposition 2.30]{CalziPeloso}). 
	We observe explicitly that, when $\Omega$ is symmetric, then $\vect 1_r\in \N_\Omega$ and the differential operator $f\mapsto f*I^{-\vect 1_r}$ is simply  differential operator associated with the determinant polynomial $\Delta^{\vect 1_r}_\Omega$ by means of the scalar product. This latter operator is often denoted by $\square$. In addition, if $\Omega$ is symmetric and irreducible, then $\N_\Omega=\Set{\vect s\in \N^r\colon s_1\Meg \cdots \Meg s_r}$, for an appropriate choice of $\Delta$. This latter condition completely determines $\Delta$ in this case.
\end{oss}

\subsection{Fourier Analysis on the \v Silov Boundary}

We now pass to the analysis of the \v Silov boundary of $D$ (cf.~\cite{OV} for a more general treatment of this topic).
We endow $E\times F_\C$ with the $2$-step nilpotent Lie group structure whose product is given by
\[
(\zeta,z)\cdot (\zeta',z')\coloneqq (\zeta+\zeta', z+z'+2 i \Phi(\zeta',\zeta)),
\]
for every $(\zeta,z),(\zeta',z')\in E\times F_\C$. If we identify $\Nc\coloneqq E\times F$ with the \v Silov boundary $\rho^{-1}(0)$ of $D$ by means of the mapping $(\zeta,x)\mapsto (\zeta,x+i\Phi(\zeta))$, then $\Nc$ becomes a $2$-step nilpotent Lie group with product
\[
(\zeta,x)(\zeta',x')=(\zeta+\zeta',x+x'+2 \Im \Phi(\zeta,\zeta'))
\]
for every $(\zeta,x),(\zeta',x')\in \Nc$.

Define $W\coloneqq\Set{\lambda\in F'\colon \exists v\in F\setminus \Set{0}\:\: \langle \lambda, \Im \Phi(\,\cdot\,,v)\rangle=0}$, so that $W$ is a proper algebraic variety in $F'$ since $\Phi$ is non-degenerate and $\overline \Omega$-positive. Then, for every $\lambda\in F'\setminus W$, the quotient of $\Nc$ modulo the central subgroup $\ker \lambda$ is isomorphic to a Heisenberg group (to $\R$, if $E=\Set{0}$), so that the Stone--Von Neuman theorem (cf., e.g.,~\cite[Theorem 1.50]{Folland}) ensures the existence of a unique (up to unitary equivalence) irreducible continuous unitary representation $\pi_\lambda$ of $\Nc$ in some Hilbert space $H_\lambda$ such that $\pi_\lambda(0,x)=\ee^{-i\langle \lambda, x \rangle}$ for every $x\in F$. One then has the Plancherel identity (cf.~\cite[Corollary 1.17 and Proposition 2.30]{CalziPeloso}):
\[
\norm{f}^2_{L^2(\Nc)}= c \int_{F'\setminus W} \norm{\pi_\lambda(f)}^2_{\Lin^2(H_\lambda)}\abs{\Delta^{-\vect b}_{\Omega'}(\lambda)}\,\dd \lambda
\]
for every $f\in L^1(\Nc)\cap L^2(\Nc)$, where $c>0$ is a suitable constant (which depends on the choice of $e_{\Omega'}$) and $\Lin^2(H_\lambda)$ denotes the space of Hilbert--Schmidt endomorphisms of $H_\lambda$. Note that $\Delta^{-\vect b}_{\Omega'}$ is positive on $\Omega'$ and extends to a polynomial on $F'$, so that the above formula is meaningful (cf.~\cite[Proposition 2.30]{CalziPeloso}).

\subsection{The CR Structure of $\Nc$}

For every $v\in E$, denote by $Z_v$ the left-invariant vector field on $\Nc$ which induces the Wirtinger derivative $\frac 1 2 (\partial_v-i \partial_{iv})$ at $(0,0)$. Then, the $Z_v$, for $v\in E$, induce a subbundle of the complexified tangent bundle of $\Nc$ which endows $\Nc$ with the structure of a CR manifold (cf.~\cite[Section 7.4]{Boggess}). In particular, a distribution $u$ on $\Nc$ is said to be CR if $\overline{Z_v}u=0$ for every $v\in E$ (cf.~\cite[Sections 9.1 and 17.2]{Boggess}). 
Note that an element $f$ of $L^2(\Nc)$ is CR is and only if
\[
\pi_\lambda(f)=\chi_{\Lambda_+}(\lambda)\pi_\lambda(f) P_{\lambda,0}
\] 
for almost every $\lambda\in F'\setminus W$, where $\Lambda_+$ is the interior of the polar of $\Phi(E)$ (that is, 
\[
\Set{\lambda\in F'\colon \forall \zeta\in E\setminus \Set{0}\:\: \langle \lambda, \Phi(\zeta)\rangle>0}\text{),}
\] and $P_{\lambda,0}$ is an orthoprojector of rank one in $H_\lambda$, for every $\lambda\in F'\setminus W$ (cf., e.g.,~\cite{OV} or~\cite[Proposition 1.19]{CalziPeloso} and~\cite[Proposition 2.6]{PWS}).

\subsection{Metrics}

We endow $D$ with a complete Riemannian metric which is invariant under the action of affine biholomorphisms (for example, the Bergman metric is complete and invariant under all biholomorphisms of $D$, cf.~\cite[Proposition 2.44]{CalziPeloso}), and the associated distance $d$. Since the balls with respect to $d$ will only be used for bounded radii, it will not matter which distance is chosen, as long as it satisfies the preceding conditions.

We endow $\Omega$ with the Riemannian metric induced by that on $D$ by means of the submersion $\rho$ (interpreted as the projection of $D$ onto its quotient modulo the action of $\Nc$), and $\Omega'$ with the metric induced by the diffeomorphism $\Omega\ni t\cdot e_\Omega \mapsto e_{\Omega'}\cdot t^{-1} \in \Omega'$. We denote by $d_\Omega$ and $d_{\Omega'}$ the corresponding distances, and by $B_\Omega(h,R)$ and $B_{\Omega'}(\lambda,R)$ the correspondings balls of centre $h\in \Omega$ and $\lambda\in \Omega'$, respectively, and radius $R>0$. 
Notice that also in this case one may choose general complete $T_+$-invariant Riemannian distances without (essentially) compromising the results which follow. Nonetheless, the relationships between $d$ and $d_\Omega$ will be useful in some places (such as in the definition of lattices given below).

Analogously, we endow $E\times \Omega$ with the the Riemannian metric induced by the one on $D$ by means of the submersion 
\[
\rho'\colon D\ni (\zeta,z)\mapsto (\zeta,\rho(\zeta,z))\in E\times \Omega,
\]
interpreted as the projection of $D$ onto its quotient modulo the action of he centre $F$ of $ \Nc$. We denote by $d_{E\times \Omega}$ the corresponding distance, and by $B_{E\times \Omega}((\zeta,h),R)$ the corresponding ball of centre $(\zeta,h)\in E\times \Omega$ and radius $R>0$. 

We observe explicitly that both $\Nc$ and its centre $F$ are normal subgroups of the group $G_\Aff$ of affine automorphisms of $D$ (cf.~\cite[Proposition 2.1]{Murakami}). Hence, $d_\Omega$ and $d_{E\times \Omega}$ are $(G_\Aff/\Nc)$- and $(G_\Aff/F)$-invariant, respectively. In particular, $d_\Omega$ and $d_{\Omega'}$ are $T_+$-invariant, while $d_{E\times \Omega}$ is invariant under the affine automorphisms of the form
\[
(\zeta,h)\mapsto (g\zeta+\zeta', t \cdot h),
\]
with $\zeta'\in E$,  $t\in T_+$, and $g\in GL(E)$ such that $t\cdot \Phi=\Phi\circ (g\times g)$. We define $\nu_{E\times \Omega}\coloneqq (\Delta_\Omega^{-\vect b-\vect d}\circ \pr_\Omega)\cdot \Hc^{2n+m}$, so that $\nu_{E\times \Omega}$ is $(G_\Aff/F)$-invariant.

\subsection{Lattices}

By a $(\delta,R)$-lattice on $\Omega$, with $\delta>0$ and $R>1$, we mean a family  $(h_k)_{k\in K}$ of elements of $\Omega$ such that the balls $B_\Omega(h_k,\delta)$ are pairwise disjoint while the balls $B_\Omega(h_k,R\delta)$ cover $\Omega$. We define lattices on $\Omega'$ and $E\times \Omega$ analogously.
Notice that every maximal family of elements of $\Omega$ whose mutual distances are $\Meg 2 \delta$ is necessarily a $(\delta,2)$-lattice (and conversely), so that $(\delta,2)$-lattices on $\Omega$, $\Omega'$, and $E\times \Omega$ always exist.

By an $\Nc$-$(\delta,R)$-lattice on $D$, with $\delta>0$ and $R>1$, we mean a family $(\zeta_{j,k},z_{j,k})_{j\in J,k\in K}$ of elements of $D$ such that the balls $B((\zeta_{j,k},z_{j,k}),\delta)$ are pairwise disjoint, the balls $B((\zeta_{j,k},z_{j,k}),R\delta)$ cover $D$, and there is a $(\delta,R)$-lattice $(h_k)_{k\in K}$ on $\Omega$ such that $\rho(\zeta_{j,k},z_{j,k})=h_k$ for every $j\in J$ and for every $k\in K$.

By an $F$-$(\delta,R)$-lattice on $D$, with $\delta>0$ and $R>1$, we mean a family $(\zeta_{k},z_{j,k})_{j\in J,k\in K}$ of elements of $D$ such that the balls $B((\zeta_{k},z_{j,k}),\delta)$ are pairwise disjoint, the balls $B((\zeta_{k},z_{j,k}),R\delta)$ cover $D$, and there is a $(\delta,R)$-lattice $(\zeta_k,h_k)_{k\in K}$ on $E\times\Omega$ such that $\rho(\zeta_k,z_{j,k})=h_k$ for every $j\in J$ and for every $k\in K$.

By a modification of the previous argument, one may show that $\Nc$- and $F$-$(\delta,4)$-lattices always exist on $D$ (cf.~\cite[Lemma 2.55]{CalziPeloso}).

\subsection{The Associated Tube Domain}\label{tube-domain:subsec}

We denote by
\[
T_\Omega=F+i \Omega
\]
the tube domain associated with $\Omega$.
Given a function $f$ on $D$, we define
\[
f_h\colon \Nc \ni(\zeta,x)\mapsto f(\zeta,x+i \Phi(\zeta)+i h)
\]
for every $h\in \Omega$, and
\[
f^{(\zeta)}\colon T_\Omega\ni z\mapsto f(\zeta,z+i\Phi(\zeta))
\]
for every $\zeta\in E$. Thus,
\[
f^{(\zeta)}_h\colon F\ni x\mapsto f(\zeta,x+i \Phi(\zeta)+i h) 
\]
for every $\zeta\in E$ and for every $h\in \Omega$. 

\section{Statement of the Main Comparison Results}\label{sec:3}

We now introduce the different definitions of mixed-norm Bergman spaces.

In~\cite{RicciTaibleson,CalziPeloso}, mixed-norm weighted Bergman spaces are defined as
\[
A^{p,q}_{\vect s}(D)=\Set{f\in \Hol(D)\colon \norm{h\mapsto\Delta_\Omega^{\vect s}(h) \norm{f_h}_{L^p(\Nc)}}_{L^q(\nu_\Omega)} <\infty }.
\]
On the one hand, this definition highlights the role played by the \v Silov boundary of $D$ and gives rise to the usual Hardy spaces when $q=\infty$ and $\vect s=\vect 0$ (that is, $\Delta_\Omega^{\vect s}=1$). In particular, the non-commutative Fourier analysis on $\Nc$ comes into play.
On the other hand, the weight $\Delta_\Omega^{\vect s}\circ \rho$ is considered as a multiplier of the function, and not of the measure, and the `base measure' is chosen in such a way that it induces the invariant measures on $\Nc$ and $\Omega$. When $q=\infty$, this allows to treat a whole class of spaces which would not appear otherwise, and which play a relevant role in the duality theory of the spaces $A^{p,q}_{\vect s}(D)$ when $q\meg 1$. 

\medskip

In~\cite{Nana,BekolleGonessaNana} (to cite only a few), mixed norm weighted Bergman spaces are defined as 
\[
\Af^{p,q}_{\vect s}(D)\coloneqq \Set{f\in \Hol(D)\colon \norm{\zeta\mapsto \norm{f^{(\zeta)}}_{A^{p,q}_{(\vect s+\vect b/2)/q}(T_\Omega) }   }_{L^q(E)}<\infty  }.
\]
On the one hand, this definition highlights the role played by the contre $F$ of the \v Silov boundary of $D$, so that the usual (commutative) Fourier analysis on $F$ comes into play. In addition, this definition also allows to think $D$ as the union of the translates $(\zeta,i\Phi(\zeta))+T_\Omega$ of the tube domain $T_\Omega$ (identified with $\Set{0}\times T_\Omega\subseteq D$), so that some of the analysis on $\Af^{p,q}_{\vect s}(D)$ may be reduced to a simpler analysis on $\Af^{p,q}_{\vect s}(T_\Omega)$.
On the other hand, the  weight $\Delta_\Omega^{\vect s}\circ \rho$ is considered as a multiplier of the `base measure' $(\Delta_\Omega^{\vect b/2+\vect d}\circ \rho )\cdot \Hc^{2n+2m}= (\Delta_\Omega^{-\vect b/2-\vect d}\circ \rho )\cdot \nu_D$, and not of the function. In this way the self-adjoint projector of $\Lf^{2,2}_{\vect s}(D)$ (defined as $\Af^{2,2}_{\vect s}(D)$, but allowing $f$ to be a meausurable function modulo negligible functions) onto $\Af^{2,2}_{\vect s}(D)$ is highlighted as the `canonical choice' when looking for a projector of $\Lf^{p,q}_{\vect s}(D)$ onto $\Af^{p,q}_{\vect s}(D)$ for different $p,q\in [1,\infty]$. 

We mention that $\Af^{p,\infty}_{\vect s}(D)=\Af^{p,\infty}_{\vect 0}(D)$ for every $\vect s\in \R^r$. Because of this fact, the case $q=\infty$ is somewhat pathological and seldom considered. For similar reasons, the duality theory for the space $\Af^{p,q}_{\vect s}(D)$, when $q\meg 1$, is treated separately (cf., e.g.,~\cite{Bekolle}).

We also observe that 
\[
\Af^{p,q}_{\vect s}(D)=\Set{f\in \Hol(D)\colon \norm{h\mapsto \norm{f_h}_{L^{p,q}(F,E)}   }_{L^q(\Delta^{\vect s+\vect b/2}_\Omega\cdot \nu_\Omega)}<\infty  },
\]
where 
\[
\norm{g}_{L^{p,q}(F,E)}\coloneqq \norm{\zeta \mapsto \norm{g(\zeta,\,\cdot\,)}_{L^p(F)}  }_{L^q(E)}
\]
for every measurable function $g\colon \Nc\to \C$.

We now compare a number of results concerning both the spaces $A^{p,q}_{\vect s}(D)$ and $\Af^{p,q}_{\vect s}$. Proofs can be found in~\cite{CalziPeloso,Paralipomena}, or in Section~\ref{sec:proofs}.

\subsection{Comparisons and Inclusions}

The following result is a consequence of~\cite[Proposition 3.5]{CalziPeloso}, Remark~\ref{oss:4}, and Lemma~\ref{oss:5}.

\begin{oss}\label{oss:6}
	Take $p,q\in (0,\infty]$ and $\vect s \in \R^r$. Then, the following hold:
	\begin{itemize}
		\item $A^{p,q}_{\vect s}(D)\neq \Set{0}$ if and only if $\vect s \succ \frac{1}{2 q }\vect m$ or $q=\infty$ and $\vect s \Meg \vect 0$;
		
		\item $\Af^{p,q}_{\vect s}(D)\neq \Set{0}$ if and only if $\vect s \succ \frac 1 2 (\vect m-\vect b)$ or $q=\infty$;
		
		\item if  $\vect s \succ \frac 1 2(\vect m-\vect b)$ or $q=\infty$, then
		\[
		A^{p,q}_{(\vect s+\vect b/2)/q}(D)=\Af^{p,q}_{\vect s}(D)
		\]
		if and only if $p=q$ or $E=\Set{0}$.\footnote{More precisely, if $p\neq q$ and $E\neq \Set{0}$, then $A^{p,q}_{(\vect s+\vect b/2)/q}(D)\not \subseteq\Af^{p,q}_{\vect s}(D)$ and $A^{p,q}_{(\vect s+\vect b/2)/q}(D)\not \supseteq\Af^{p,q}_{\vect s}(D)$.}
	\end{itemize}
\end{oss}

The following result is a consequence of~\cite[Proposition 3.2]{CalziPeloso} and Proposition~\ref{prop:2}. The first assertion is contained in~\cite[Proposition 2.2]{RicciTaibleson} when $D=\C_+$, while the second assertion is contained in~\cite[Lemma 2.1]{Rochberg} when $p_1=q_1$ and $p_2=q_2$, and is essentially contained in~\cite[Proposition 3.22]{Bekolleetal} when $n=0$. 

\begin{prop}
	Take $p_1,p_2,q_1,q_2\in (0,\infty]$ with $p_1\meg p_2, q_1\meg q_2$ and $\vect{s_1},\vect{s_2}\in \R^r$. Then, the following hold:
	\begin{itemize}
		\item $A^{p_1,q_1}_{\vect{s_1}}(D)\subseteq A^{p_2,q_2}_{\vect{s_2}}(D)$ continuously, provided that $\vect{s_2}=\vect{s_1}+\left(\frac{1}{p_2}-\frac{1}{p_1}  \right)(\vect b+\vect d)$;
		
		\item $\Af^{p_1,q_1}_{\vect{s_1}}(D)\subseteq \Af^{p_2,q_2}_{\vect{s_2}}(D)$ continuously, provided that $q_2<\infty$ and $\vect{s_2}=\frac{q_2}{q_1}\vect{s_1}+\left(\frac{q_2}{p_2}-\frac{q_2}{p_1}  \right)\vect d+\left(\frac{1}{2}-\frac{q_2}{2 q_1}  \right)\vect b$.
	\end{itemize}
\end{prop}

\subsection{Reproducing Kernels}

Define the auxiliary function
\[
B^{\vect s}_{(\zeta',z')}(\zeta,z)\coloneqq \Delta^{\vect s}_{\Omega}\left( \frac{z-\overline{z'}}{2 i}-\Phi(\zeta,\zeta')  \right)
\]
for every $(\zeta,z),(\zeta',z')\in (D\times\overline D)\cup(\overline D\times D)$, where $\overline D$ denotes the closure of $D$ in $E\times F_\C$ (note that conjugation on $E$ is \emph{not} defined).

Then, by~\cite[Theorem 5.4]{Gindikin} and~\cite[Theorem II.6]{BekolleKagou} (cf., also,~\cite[Proposition 3.11]{CalziPeloso}) and Remark~\ref{oss:6}, the following result hold.

\begin{prop}
	If $\vect s \succ \frac 1 4 \vect m$, then $A^{2,2}_{\vect s}(D)$ is a reproducing kernel Hilbert space with reproducing kernel
	\[
	K^{\vect s}\colon ((\zeta,z),(\zeta',z'))\mapsto c_{\vect s} B^{\vect b+\vect d-2 \vect s}_{(\zeta',z')}(\zeta,z)
	\]
	for a suitable $c_{\vect s}\neq 0$. Analogously, if $\vect s \succ \frac 1 2 \vect m$, then $\Af^{2,2}_{\vect s}(D)$ is a reproducing kernel Hilbert space with reproducing kernel
	\[
	\Kf^{\vect s}\colon ((\zeta,z),(\zeta',z'))\mapsto c_{\vect s/2+\vect b/4} B^{\vect b/2+\vect d- \vect s}_{(\zeta',z')}(\zeta,z).
	\]
\end{prop}

In~\cite{CalziPeloso}, for notational convenience the corresponding integral operators are based on $B^{\vect s}$ rather than $K^{\vect s}$, so that the operators
\[
P_{\vect s}\colon f \mapsto c_{(\vect b+\vect d-\vect s)/2} \int_D f(\zeta,z) B^{\vect s}_{(\zeta,z)}\Delta_\Omega^{-\vect s}(\rho(\zeta,z))\,\dd \nu_D(\zeta,z)
\]
are considered. 

Then, 
\[
\Pf_{\vect s}= P_{\vect b/2+\vect d-\vect s}\colon f \mapsto c_{\vect s/2+\vect b/4}\int_D f(\zeta,z) \Kf^{\vect s}_{(\zeta,z)} \Delta^{\vect s+\vect b/2+\vect d}_\Omega(\rho (\zeta,z))\,\dd (\zeta,z)
\]
is the orthoprojector of $\Lf^{2,2}_{\vect s}(D)$ onto $\Af^{2,2}_{\vect s}(D)$.

The following result is a consequence of~\cite[Proposition 3.13]{CalziPeloso} and Proposition~\ref{prop:3}. The first assertion is contained in~\cite[Theorem 3.1]{RicciTaibleson} when $D=\C_+$. 

\begin{prop}
	Take $p,q\in (0,\infty]$ and $\vect s,\vect s'\in \R^r$.
	If:
	\begin{itemize}
		\item $\vect s\succ \frac 1 p (\vect b+\vect d)+\frac{1}{2 q'}\vect m'$;
		
		\item $\vect s'\prec \frac{1}{p'} (\vect b+\vect d)-\frac{1}{2 p'}\vect m',\vect b+\vect d-\frac 1 2 \vect m$;
		
		\item $\vect s+\vect s'\prec \frac{1}{\min(1,p)}(\vect b+\vect d)-\frac{1}{2 q'}\vect m'$ or $q'=\infty$ and $\vect s+\vect s'\meg \frac{1}{\min(1,p)}(\vect b+\vect d)$;
	\end{itemize}	
	then $P_{\vect s'}f=f$  for every $f\in A^{p,q}_{\vect s}(D)$.
	
	If $q<\infty$ and:
	\begin{itemize}
%		\item $\vect s,\vect s'\succ \frac 1 2 (\vect m-\vect b)$;
%		
%		\item $p'_{\vect s'}<p'$
%		
%		\item $\widetilde Q_{1,\vect s}(p)<q<Q_{1,\vect s}(p)$;
%		
%		\item $\widetilde Q_{2,\vect s,\vect s'}(p)\meg q<Q_{2,\vect s,\vect s'}(p)$;
		\item $\vect s\succ \frac 1 2 \vect b+\frac q p \vect d+\frac{q}{2q'}\vect m'$;
		
		\item $\vect s'\succ \frac 1 2 \vect b +\frac{1}{\max(1,p)}\vect d+\frac{1}{2p'}\vect m',\frac 1 2 (\vect m-\vect b)$;
		
		\item $\vect s'-\frac 1 q\vect s\succ -\frac{1}{2 q'}\vect b-\left(\frac 1 p-1\right)_+\vect d+\frac{1}{2 q'}\vect m$ or $q'=\infty$ and $\vect s'-\frac 1 q\vect s\Meg -\left(\frac{1}{2 q}-\frac 1 2\right)\vect b-\left(\frac 1 p-1\right)_+\vect d $;
	\end{itemize} 
	then $\Pf_{\vect s'}f=f$  for every $f\in \Af^{p,q}_{\vect s}(D)$.
\end{prop}

%In the particular case $\vect s=\vect{s'}$ and $p,q\in (1,\infty)$, the latter conditions become:
%\begin{itemize}
%	\item  
%	$\vect s\succ \frac 1 2 \vect b+\frac q p \vect d+\frac{q-1}{2}\vect m',\frac 1 2 \vect b +\frac{1}{\max(1,p)}\vect d+\frac{1}{2p'}\vect m',\frac 1 2 (\vect m-\vect b), \frac{1}{2(q-1) }\vect b-q'\left(\frac 1 p-1\right)_+\vect d+\frac{1}{2 }\vect m$.
%\end{itemize}

\subsection{Sampling}

The following sampling theorems are consequences of~\cite[Theorem 3.22]{CalziPeloso} and Theorem~\ref{teo:1}, where more precise versions of these results are proved. The second assertion is contained in~\cite[Theorem 5.6]{Bekolleetal} when $n=0$, $p=q$, and $\Omega$ is symmetric, in~\cite[Theorem 5.2]{BekolleIshiNana} when $p=q$, and in~\cite[Theorem 3.3]{BekolleGonessaNana2} when $n=0$ and $\Omega$ is symmetric. We denote by $\ell^{p,q}(J,K)$ the space of $\lambda\in \C^{J\times K}$ such that $\norm{\norm{\lambda_{j,k}}_{\ell^p(J)}}_{\ell^q(K)}<\infty$, with some abuse of notation.

\begin{teo}
	Take $p,q\in (0,\infty]$, $\vect s\in \R^r$ and $R_0>1$. Then, there is $\delta_0>0$ such that, for every $\Nc$-$(\delta,R)$-lattice $(\zeta_{j,k},z_{j,k})_{j\in J,k\in K}$ on $D$, with $\delta\in (0,\delta_0]$ and $R\in (1,R_0]$,  the mapping 
	\[
	f \mapsto \Delta_\Omega^{\vect s-(\vect b+\vect d)/p}(\rho(\zeta_{j,k},z_{j,k})) f(\zeta_{j,k},z_{j,k})
	\]
	induces an isomorphism of $A^{p,q}_{\vect s}(D)$ onto a closed subspace of $\ell^{p,q}(J,K)$, and such that, for every $F$-$(\delta',R')$-lattice $(\zeta'_{k'},z'_{j',k'})_{j'\in J',k'\in K'}$  on $D$, with $\delta'\in (0,\delta_0]$ and $R'\in (1,R_0]$,  the mapping 
	\[
	f \mapsto \Delta_\Omega^{\vect s/q-\vect b/(2q)-\vect d/p}(\rho(\zeta'_{k'},z'_{j',k'})) f(\zeta'_{k'},z'_{j',k'})
	\]
	induces an isomorphism of $\Af^{p,q}_{\vect s}(D)$ onto a closed subspace of $\ell^{p,q}(J,K)$.
\end{teo}

Here we mention that the transpose of the sampling maps defined above is often considered an atomic decomposition map, especially when the duals of $A^{p,q}_{\vect s}(D)$ and $\Af^{p,q}_{\vect s}(D)$ may be identified with $A^{p',q'}_{\vect s'}(D)$ and $\Af^{p',q'}_{\vect s'}(D)$, respectively, for some $\vect s'$. See~\cite{CalziPeloso,Paralipomena} and Subsection~\ref{sec:atomic} and Theorem~\ref{teo:3} for more information. Since, in particular, the problem of determining the validity of the aforementioned atomic decomposition and duality is equivalent (when $p,q\in [1,\infty]$) to the problem of determining the range of boundedness of the Bergman projectors, we shall only discuss the latter problem in this section.
%[\textbf{COMMENTO:} se occorre, si può inserire anche una sezione su decomposizione atomica e dualità.]

\subsection{Boundary Values}

We now consider the problem of determining the boundary values of the various weighted Bergman spaces.
We recall some definitions and results from~\cite{CalziPeloso}, with some slight changes motivated by~\cite{Besov,PWS}.

\begin{deff}
	We define $\Sc_{\overline{\Omega'}}(\Nc)$  as the space of CR $\phi\in \Sc(\Nc)$ such that  $\Fc_F(\phi(\zeta,\,\cdot\,))$ is supported in $\overline{\Omega'}$ for every $\zeta\in E$, endowed with the topology induced by $\Sc(\Nc)$. We define $\Sc_{\overline{\Omega'}}'(\Nc)$ as the dual of the conjugate of $\Sc_{\overline{\Omega'}}(\Nc)$.
	In addition, we define $\widetilde \Sc_{\overline{\Omega'}}(\Nc)$ as the space of $\phi\in \Sc(\Nc)$ such that $\pi_\lambda(\phi)=\chi_{\Omega'}(\lambda) P_{\lambda,0}\pi_\lambda(\phi)P_{\lambda,0}$ for every $\lambda\in F'\setminus W$.
	We define
	\[
	\Fc_\Nc\colon \widetilde \Sc_{\overline{\Omega'}}(\Nc)\ni \phi\mapsto [\lambda \mapsto \tr(\pi_\lambda(\phi))].
	\]
\end{deff}

Notice that $\Sc_{\overline{\Omega'}}(\Nc)$ may be equivalently defined as the set of $\phi\in \Sc(\Nc)$ such that $\pi_\lambda(\phi)=\chi_{\Omega'}(\lambda) \pi_\lambda(\phi)P_{\lambda,0}$ for every $\lambda\in F'\setminus W$, thanks to~\cite[Proposition 2.7]{PWS}.
In addition, $\Fc_\Nc$ induces an isomorphism of $\widetilde \Sc_{\overline{\Omega'}}(\Nc)$ onto the space of Schwartz functions on $F'$ supported in $\overline{\Omega'}$ (cf.~\cite[Proposition 5.2]{PWS}).

\begin{deff}\label{def:2}
	Take $p,q\in (0,\infty]$ and $\vect s\in \R^r$. We define $B^{\vect s}_{p,q}(\Nc,\Omega)$ as the space of $u\in \Sc'_{\overline{\Omega'}}(\Nc)$ such that
	\[
	(\Delta_{\Omega'}^{\vect s}(\lambda_k) u*\psi_k)\in \ell^q(K;L^p(\Nc)),
	\]
	where $(\lambda_k)_{k\in K}$ is a $(\delta,R)$-lattice on $\Omega'$ and $(\psi_k)$ is a family of elements of $\widetilde \Sc_{\overline{\Omega'}}(\Nc)$ such that $((\Fc_\Nc \psi_k)(\,\cdot\, t_k))$ is a bounded family of positive elements of $C^\infty_c(\Omega')$,\footnote{Notice that this means that the $(\Fc_\Nc \psi_k)(\,\cdot\, t_k)$ are supported in a fixed compact subset of $\Omega'$, and are uniformly bounded with every derivative.} with $\lambda_k=e_{\Omega'}\cdot t_k$, and 
	\[
	\sum_k \Fc_\Nc \psi_k\Meg 1
	\]
	on $\Omega'$.
	
	We define $\Bf^{\vect s}_{p,q}(\Nc,\Omega)$ as the space of $u\in \Sc'_{\overline{\Omega'}}(\Nc)$ such that
	\[
	(\Delta^{(\vect s-\vect b/2)/q}_{\Omega'}(\lambda_k) u*\psi_k)\in \ell^q(K; L^{p,q}(F,E)),
	\]
	endowed with the corresponding topology. 
\end{deff}

The definition of $B^{p,q}_{\vect s}(\Nc,\Omega)$ does not depend on the choice of $\delta$, $R$, $(\lambda_k)$, and $(\psi_k)$ (cf.~\cite[Lemma 4.14]{CalziPeloso}). In addition, $B^{p,q}_{\vect s}(\Nc,\Omega)$ is a quasi-Banach space (cf.~\cite[Proposition 4.16]{CalziPeloso} and~\cite[Proposition 7.12]{Besov}). Analogous assertions hold for $\Bf^{p,q}_{\vect s}(\Nc,\Omega)$ (cf.~Proposition~\ref{prop:12}).

Observe that, by~\cite[the remarks following the statement of Lemma 5.1]{CalziPeloso}, there is a constant $c>0$ such that
\[
f(\zeta,z)=c \int_\Nc f_0(\zeta,x) B^{\vect b+\vect d}_{(\zeta,x+i\Phi(\zeta))}(\zeta,z)\,\dd (\zeta,x)
\]
for every $f\in H^2(D)=A^{2,\infty}_{\vect 0}(D)$ and for every $(\zeta,z)\in D$, where $f_0$ is the limit of $(f_h)$ in $L^2(\Nc)$ for $h\to 0$, $h\in \Omega$. In other words, 
\[
S_{(\zeta,z)}\coloneqq c  \left( B^{\vect b+\vect d}_{(\zeta,z)}\right)_0
\]
is (the boundary values of) the Cauchy--Szeg\H o kernel. 

The following result is a consequence of~\cite[Proposition 4.20, Theorem 4.23, and Lemma 5.1]{CalziPeloso}. 

\begin{prop}
	Take $(\lambda_k)$ and $(\psi_k)$ as in Definition~\ref{def:2}, in such a way that $\sum_k (\Fc_\Nc \psi_k)^2=1$ on $\Omega'$. Then, there is a  continuous sesquilinear form
	\[
	\langle \,\cdot\,\vert \,\cdot \rangle\colon B^{\vect s}_{p,q}(\Nc,\Omega)\times B^{-\vect s-(1/p-1)_+(\vect b+\vect d)}_{p',q'}(\Nc,\Omega)\ni (u,u')\mapsto \sum_k \langle u*\psi_k\vert u'*\psi_k\rangle\in \C
	\]
	which induces an antilinear isomorphism of $B^{-\vect s-(1/p-1)_+(\vect b+\vect d)}_{p',q'}(\Nc,\Omega)$ onto the dual of the closure of  $\Sc_{\overline{\Omega'}}(\Nc)$ in $B^{\vect s}_{p,q}(\Nc,\Omega)$. 
	
	In addition, $S_{(\zeta,z)}\in B^{-\vect s-(1/p-1)_+(\vect b+\vect d)}_{p',q'}(\Nc,\Omega)$ for every $(\zeta,z)\in D$ if $\vect s\succ\frac{1}{p}(\vect b+\vect d)+\frac{1}{2 q'}\vect{m'}$.
\end{prop}

A similar result holds for the spaces $\Bf^{\vect s}_{p,q}(\Nc,\Omega)$ as well (cf.~Lemma~\ref{lem:4} and the remarks following the statement of Proposition~\ref{prop:7}). Nonetheless, because of the pathological behaviour of the spaces $\Bf^{\vect s}_{p,q}(\Nc,\Omega)$ for $q=\infty$, it is not possible to state the analogous result here without defining additional spaces. We therefore refer the reader to the aforementioned results for more precise information.

%By Proposition~\ref{prop:10}, for every $\zeta\in E$ there is a continuous linear mapping $\Bf^{\vect s}_{p,q}(\Nc,\Omega)\ni u \mapsto u(\zeta,\,\cdot\,)\in B^{(\vect s+\vect b/2)/q}_{p,q}(F,\Omega)$. This allows us to state the following definition.

\begin{deff}
	Given $\vect s\succ\frac{1}{p}(\vect b+\vect d)+\frac{1}{2 q'}\vect{m'} $, we define a continuous linear operator
	\[
	\Ec \colon B^{-\vect s}_{p,q}(\Nc,\Omega)\ni u\mapsto [(\zeta,z)\mapsto \langle u\vert S_{(\zeta,z)}\rangle]\in A^{\infty,\infty}_{\vect s-(\vect b+\vect d)/p}(D),
	\]
	and denote by $\widetilde A^{p,q}_{\vect s}(D)$ its image, endowed with the corresponding topology.
	
	If $q<\infty$, given $\vect s\succ \frac 1 2 \vect b+\frac q p \vect d+\frac{q}{2q'}\vect m'$, we define a continuous linear mapping 
	\[
	\Ec \colon \Bf^{-\vect s}_{p,q}(\Nc,\Omega)\ni u\mapsto [(\zeta,z)\mapsto \langle u\vert S_{(\zeta,z)}\rangle]\to A^{\infty,\infty}_{(\vect s-\vect b/2)/ q-\vect d/p}.
	\] 
	We denote by $\widetilde \Af^{p,q}_{\vect s}(D)$ the image of $\Bf^{-\vect s}_{p,q}(\Nc,\Omega)$ under $\Ec$, endowed with the corresponding topology.
\end{deff}

%The two definitions of $\Ec$ agree by Proposition~\ref{prop:11}. In addition,
Notice that $(\Ec u)_h\to u$ in $\Sc'_{\overline{\Omega'}}(\Nc)$ for $h\to 0$, $h\in \Omega'$, for every $u\in B^{-\vect s}_{p,q}(\Nc,\Omega)$ and for every $u\in\Bf^{-\vect s}_{p,q}(\Nc,\Omega)$  (cf.~\cite[Theorem 5.2]{CalziPeloso},~\cite[Proposition 7.12]{Besov}, and Proposition~\ref{prop:11}), so that $\Ec$ is one-to-one and $ B^{-\vect s}_{p,q}(\Nc,\Omega)$ and $\Bf^{-\vect s}_{p,q}(\Nc,\Omega) $ are the spaces of boundary values of $\widetilde A^{p,q}_{\vect s}(D)$ and $\widetilde \Af^{p,q}_{\vect s}(D)$, respectively (when defined).

The following result is a consequence of~\cite[Proposition 5.4 and Corollary 5.11]{CalziPeloso}, and of Proposition~\ref{prop:6}. The second assertion is essentially contained in~\cite[Theorem 1.7]{BekolleBonamiGarrigosRicci} when $p,q\Meg 1$, $n=0$,  $\Omega$ is symmetric, and $\vect s\in \R\vect d$, and in~\cite[Corollary 4.7]{Debertol} when $p,q> 1$, $n=0$, and  $\Omega$ is symmetric.
Notice that the inclusion $\Ec(\Sc_{\overline{\Omega'}}(\Nc))\subseteq A^{p,q}_{\vect s}(D)$ does not follow from~\cite[Proposition 5.4]{CalziPeloso}, but follows easily from~\cite[Theorem 4.2]{PWS}.

\begin{teo}\label{teo:4}
	Take $p,q\in (0,\infty]$ and $\vect s\in \R^r$. 
	Then, the following hold:
	\begin{enumerate}
		\item[\textnormal{(1)}] if $\vect s\succ\frac{1}{p}(\vect b+\vect d)+\frac{1}{2 q'}\vect{m'}$ and either $\vect s\succ \frac{1}{2 q}\vect m$ or $q=\infty$ and $\vect s\Meg \vect 0$, then
		\[
		\Ec(\Sc_{\overline{\Omega'}}(\Nc))\subseteq A^{p,q}_{\vect s}(D)\subseteq \widetilde A^{p,q}_{\vect s}(D)
		\]
		continuously, with equality in the second inclusion if 
		\[
		\vect s \succ \frac{1}{2 q}\vect m+\left( \frac{1}{2 \min(p,p')}-\frac{1}{2 q}\right)_+\vect{m'};
		\]
		
		\item[\textnormal{(2)}] if $q<\infty$ and $\vect s  \succ \frac 1 2 \vect b+\frac q p \vect d+\frac{q}{2q'}\vect m',\frac 1 2 (\vect m-\vect b)$, then
		\[
		\Ec(\Sc_{\overline{\Omega'}}(\Nc))\subseteq \Af^{p,q}_{\vect s}(D)\subseteq \widetilde \Af^{p,q}_{\vect s}(D)
		\]
		continuously, with equality in the second inclusion if 
		\[
		\vect s \succ \frac{1}{2}(\vect m-\vect b)+\left( \frac{q}{2 \min(p,p')}-\frac{1}{2 }\right)_+\vect{m'}.
		\]
	\end{enumerate}
\end{teo}

\begin{oss}
We observe explicitly that $\vect s \succ \frac{1}{2}(\vect m-\vect b)+\left( \frac{q}{2 \min(p,p')}-\frac{1}{2 }\right)_+\vect{m'}$ if and only if $\vect s \succ \frac{1}{2}(\vect m-\vect b)$ and
\[
q<q_{\vect s}(p)\coloneqq \min(p,p')\min_{j=1,\dots, r} \frac{2s_j+ b_j+m'_j- m_j}{m'_j}.
\]
\end{oss}

We also have transference results (cf.~\cite[Theorem 6.3]{Paralipomena} and Proposition~\ref{prop:6}).  

\begin{prop}
	Take $p,q\in (0,\infty]$ and $\vect s\in \R^r$. 
	Then the following hold:
	\begin{enumerate}
		\item[\textnormal{(1)}] if   $\frac 1 q\vect s\succ-\frac{1}{2 q} \vect b+\frac 1 p \vect d+\frac{1}{2q'}\vect{m'} $ and $\Af^{p,q}_{\vect s+\vect b/2}(T_\Omega)=\widetilde \Af^{p,q}_{\vect s+\vect b/2}(T_\Omega)$, then $\Af^{p,q}_{\vect s}(D)=\widetilde \Af^{p,q}_{\vect s}(D)$;
		
		\item[\textnormal{(2)}] if $\vect s \succ \frac{1}{p}(\vect b+\vect d)+\frac{1}{2 q'}\vect{m'}$ and $A^{p,q}_{\vect s}(D)=\widetilde A^{p,q}_{\vect s}(D)$, then $A^{p,q}_{\vect s-\vect b/p}(T_\Omega)=\widetilde A^{p,q}_{\vect s-\vect b/p}(T_\Omega)$. 
	\end{enumerate}
\end{prop}

\subsection{Bergman Projectors}

Concerning the boundedness of Bergman projectors, we have the following results.

\begin{prop}
	Take $p,q\in [1,\infty]$ and $\vect s,\vect s' \in \R^r$. Then, the following hold:
	\begin{enumerate}
	\item[\textnormal{(1)}] if $\vect s'\prec \vect b+\vect d-\frac 1 2 \vect m$ and $P_{\vect{s'}}$ induces a continuous linear projector of $L^{p,q}_{\vect s}(D)$ onto $A^{p,q}_{\vect s}(D)$, then:
	\begin{itemize}
		\item $\vect s \succ\frac 1 p(\vect b+\vect d)+\frac{1}{2 q'}\vect{m'} $, and $\vect s\succ \frac{1}{2 q}\vect m$ or $q=\infty$ and $\vect s \Meg \vect 0$;
		
		\item $\vect{s'}\prec \frac{1}{\min(p,p')}(\vect b+\vect d-\frac 1 2\vect{m'})$;
		
		\item $\vect s+\vect s'\prec \frac 1 p (\vect b+\vect d)-\frac{1}{2 q}\vect m'$, and $\vect s+\vect s'\prec \vect b+\vect d-\frac{1}{2 q'}\vect m$ or $q=1$ and $\vect s+\vect s'\meg \vect b+\vect d$.
	\end{itemize}
	
	\item[\textnormal{(2)}] if $q<\infty$, $\vect s'\succ \frac 1 2 (\vect m-\vect b)$, and $\Pf_{\vect{s'}}$ induces a continuous linear projector of $\Lf^{p,q}_{\vect s}(D)$ onto $\Af^{p,q}_{\vect s}(D)$, then:
	\begin{itemize}
		\item $\vect s\succ \frac 1 2 (\vect m-\vect b), \frac 1 2 \vect b+\frac q p \vect d+\frac{q-1}{2}\vect{m'}$;
		
		\item $\vect{s'}\succ \frac 1 2 \vect b+\frac{1}{\max(p,p')}\vect d+\frac{1}{2 \min(p,p')}\vect{m'}$;
		
		\item $\vect{s'}-\frac1 q\vect s\succ  \frac{1}{2 q'} \vect b+\frac{1}{p'}\vect d+\frac{1}{2 q}\vect{m'}$ and $\vect{s'}-\frac1 q\vect s\succ  -\frac{1}{2 q'} \vect b+\frac{1}{2 q'}\vect{m}$ or $q=1$ and $\vect{s'}-\vect s\Meg  \vect 0$.
	\end{itemize}
	\end{enumerate}
\end{prop}

In particular, if $\vect s\succ\frac 1 2 (\vect m-\vect b)$, $p,q\in (1,\infty)$, and $\Pf_{\vect{s}}$ induces a continuous linear projector of $\Lf^{p,q}_{\vect s}(D)$ onto $\Af^{p,q}_{\vect s}(D)$, then:
\[
p_{\vect s}'<p<p_{\vect s} \qquad\text{and} \qquad Q'_{\vect s}(p')<q<Q_{\vect s}(p),
\]
where
\[
p_{\vect s}\coloneqq \min_{j=1,\dots, r} \frac{ m'_j-2 d_j }{(b_j+m'_j-2 s_j)_+} \qquad \text{and} \qquad Q_{\vect s}(p)\coloneqq \min_{j=1,\dots,r} \frac{s_j-\frac 1 2 b_j+\frac 1 2 m'_j}{\left( \frac 1 p d_j+\frac 1 2 m'_j \right)_+}.
\]
%
%\[
%\vect s\succ \frac 1 2 (\vect m-\vect b), \frac 1 2 \vect b+\frac q p \vect d+\frac{q-1}{2}\vect{m'},\frac{1}{2} \vect b+\frac{q'}{p'}\vect d+\frac{1}{2 (q-1)}\vect{m'}, \frac 1 2 \vect b+\frac{1}{\max(p,p')}\vect d+\frac{1}{2 \min(p,p')}\vect{m'}.
%\]

There are also transference results (cf.~\cite{Paralipomena,BekolleGonessaNana} and Corollary~\ref{cor:2}).

\begin{prop}
	Take $p,q\in [1,\infty]$ and $\vect s,\vect s'\in \R^r$.
	Then, the following hold:
	\begin{itemize}
		\item if $\vect{s'}\prec \vect b+\vect d-\frac 1 2 \vect m$ and $P_{\vect{s'}}$ induces a continuous linear projector of $L^{p,q}_{\vect s}(D)$ onto $A^{p,q}_{\vect s}(D)$, then $P_{\vect{s'}}$ induces a continuous linear projector of $L^{p,q}_{\vect s-\vect b/p}(T_\Omega)$ onto $A^{p,q}_{\vect s-\vect b/p}(T_\Omega)$;
		
		\item if $\vect{s'}\succ \frac 1 2 (\vect m-\vect b)$  and $\Pf_{\vect{s'}+\vect b/2}$ induces a continuous linear projector of $\Lf^{p,q}_{\vect s+\vect b/2}(T_\Omega)$ onto $\Ac^{p,q}_{\vect s+\vect b/2}(T_\Omega)$, then $\Pf_{\vect{s'}}$ induces a continuous linear projector of $\Lf^{p,q}_{\vect s}(D)$ onto $\Af^{p,q}_{\vect s}(D)$.
	\end{itemize}
\end{prop}

The following result is a consequence of~\cite[Corollary 5.27]{CalziPeloso} (cf.~also~\cite[Corollary 4.7]{Paralipomena}), Proposition~\ref{prop:6}, and Theorem~\ref{teo:3}. The second assertion is contained in~\cite[Theorem 1.9]{BekolleBonamiGarrigosRicci} when $p,q\Meg 1$, $\vect s=\vect s'$, $\Omega$ is symmetric, and $\vect s\in \R\vect d$, and in~\cite[Corollary 1.4]{Debertol} when $p,q\Meg 1$, $\vect s=\vect s'$, and $\Omega$ is symmetric.

\begin{teo}
	Take $p,q\in [1,\infty]$ and $\vect s,\vect s'\in \R^r$. Then, the following hold:
	\begin{enumerate}
		\item[\textnormal{(1)}] if:
	\begin{itemize}
		\item $\vect s\succ \frac{1}{2 q}\vect m+\left( \frac{1}{2 \min(p,p')}-\frac{1}{2 q} \right)_+\vect{m'}$;
		
		\item $\vect b+\vect d-\vect s-\vect{s'}\succ\frac{1}{2 q'}\vect m+\left( \frac{1}{2 \min(p,p')}-\frac{1}{2 q'} \right)_+\vect{m'}$;
	\end{itemize}
	then $P_{\vect{s'}}$ induces a continuous linear projector of $L^{p,q}_{\vect s}(D)$ onto $A^{p,q}_{\vect s}(D)$;
	
	\item[\textnormal{(2)}] if $q<\infty$ and:
	\begin{itemize}
		\item $\vect s\succ \frac{1}{2}(\vect m-\vect b)+\left( \frac{q}{2 \min(p,p')}-\frac{1}{2} \right)_+\vect{m'}$;
		
		\item $\vect{s'}-\frac 1 q\vect s\succ\frac{1}{2 q'}(\vect m-\vect b)+\left( \frac{1}{2 \min(p,p')}-\frac{1}{2 q'} \right)_+\vect m'$;
	\end{itemize}
	then $\Pf_{\vect{s'}}$ induces a continuous linear projector of $\Lf^{p,q}_{\vect s}(D)$ onto $\Af^{p,q}_{\vect s}(D)$.
	\end{enumerate}
\end{teo}

In particular, if $\vect s=\vect{s'}\succ\frac 1 2 (\vect m-\vect b)$ and 
\[
q_{\vect s}'(p)<q<q_{\vect s}(p),
\]
then  $\Pf_{\vect{s'}}$ induces a continuous linear projector of $\Lf^{p,q}_{\vect s}(D)$ onto $\Af^{p,q}_{\vect s}(D)$ (cf.~the remarks after the statement of Theorem~\ref{teo:4}).

\section{An Auxiliary Space}\label{sec:proofs}

For technical reasons, we shall  consider the spaces
\[
\Ac^{p,q}_{\vect s}(D)\coloneqq\Set{f\in \Hol(D)\colon \norm{h\mapsto \norm{f^{(\zeta)}}_{A^{p,q}_{\vect s}(T_\Omega)}}_{L^q(E)} <\infty }
\]
and
\[
\Ac^{p,q}_{\vect s,0}(D)\coloneqq\Hol(D)\cap \Lc^{p,q}_{\vect s,0}(D),
\]
where $\Lc^{p,q}_{\vect s,0}(D)$ denotes the closure of $C_c(D)$ in $\Lc^{p,q}_{\vect s}(D)$ (defined as $\Ac^{p,q}_{\vect s}(D)$ replacing $\Hol(D)$ with the space of measurable functions modulo negligible functions).

\begin{oss}\label{oss:4}
	Take $p,q\in(0,\infty]$ and $\vect s\in \R^r$.  Then, $\Ac^{p,q}_{\vect s,0}(D)\neq \Set{0}$ (resp.\ $\Ac^{p,q}_{\vect s}(D)\neq \Set{0}$)  if and only if $\vect s\succ \frac{1}{2 q}\vect m$ (resp.\ $\vect s\Meg \vect 0$ if $q=\infty$).
\end{oss}

\begin{proof}
	Argue as in~\cite[Proposition 3.5]{CalziPeloso}.
\end{proof} 

\begin{lem}\label{oss:5}
	Take $p,q\in(0,\infty]$ and $\vect s\in \R^r$. Then,
	\[
	\Af^{p,q}_{\vect s}(D)=\Ac^{p,q}_{(\vect s+\vect b/2)/q}(D).
	\]
	In addition, if either $\vect s \succ \frac {1}{2 q}\vect m$, or $q=\infty$ and $\vect s \Meg \vect 0$, then
	\[
	A^{p,q}_{\vect s}(D)=\Ac^{p,q}_{\vect s}(D)
	\]
	if and only if $p=q$ or $E=\Set{0}$.
\end{lem}

Note that, as the proof shows, if $p\neq q$ and $E\neq \Set{0}$, then $\Ac^{p,q}_{\vect s}(D)\not \subseteq A^{p,q}_{\vect s}(D)$ and $\Ac^{p,q}_{\vect s}(D)\not \supseteq A^{p,q}_{\vect s}(D)$.

\begin{proof}
	The first assertion is clear. As for what concerns the second assertion,
	it is clear that $A^{p,q}_{\vect s}(D)=\Ac^{p,q}_{\vect s}(D)$ if $p=q$ (by Fubini's theorem) and if $E=\Set{0}$. Conversely, assume that $p\neq q$ and that $E\neq \Set{0}$, so that $\vect b\neq 0$. Observe that, given $t\in T_+$ and $g\in GL(E)$ such that $t\cdot \Phi=\Phi\circ(g\times g)$ and $f\in \Hol(D)$, one has
	\[
	\norm{f\circ (g\times t)}_{A^{p,q}_{\vect s}(D)}=\Delta^{(\vect b+\vect d)/p-\vect s}(t) \norm{f}_{A^{p,q}_{\vect s}(D)}
	\]
	and
	\[
	\norm{f\circ (g\times t)}_{\Ac^{p,q}_{\vect s}(D)}=\Delta^{\vect b/q+\vect d/p-\vect s}(t)\norm{f}_{\Ac^{p,q}_{\vect s}(D)}
	\]
	so that, letting $t\to \infty$, we see that the norms on the quasi-Banach spaces $\Ac^{p,q}_{\vect s}(D)$ and $A^{p,q}_{\vect s}(D)$ cannot be comparable. The assertion follows from the open mapping and the closed graph theorems.
\end{proof}

\begin{prop}\label{prop:2}
	Take $p_1, p_2, q_1, q_2\in(0,\infty]$ and $\vect{s_1},\vect{s_2}\in \R^r$. If
	\[
	p_1\meg p_2, \qquad q_1\meg q_2, \qquad \text{and}\qquad \vect{s_2}=\vect{s_1}+\left(\frac{1}{p_2}-\frac{1}{p_1}  \right)\vect d+\left(\frac{1}{q_2}-\frac{1}{q_1}  \right)\vect b,
	\]
	then $\Ac^{p_1,q_1}_{\vect{s_1}}(D)\subseteq\Ac^{p_2,q_2}_{\vect{s_2}}(D)$ and $\Ac^{p_1,q_1}_{\vect{s_1},0}(D)\subseteq\Ac^{p_2,q_2}_{\vect{s_2},0}(D)$.
	
	In addition, the mappings $\Ac^{p_1,q_1}_{\vect{s_1}}(D)\ni f \mapsto f^{(\zeta)}\in A^{p_1,q_1}_{\vect{s_1}-\vect b/q_1}(T_\Omega)$, as $\zeta$ runs through $E$, are equicontinuous and map $\Ac^{p_1,q_1}_{\vect{s_1},0}(D)$ into $A^{p_1,q_1}_{\vect{s_1}-\vect b/q_1,0}(T_\Omega)$.
\end{prop}

\begin{proof}
	Argue as in~\cite[Propositions 3.2 and 3.7]{CalziPeloso} and~\cite[Proposition 6.5]{Paralipomena}.	
\end{proof}

\begin{cor}\label{oss:1}
	Take $p,q\in(0,\infty]$, $\vect s\in \R^r$, and $f\in \Ac^{p,q}_{\vect s}(D)$. Then, the function $h\mapsto \norm{f^{(\zeta)}_h}_{L^p(F)}$ is decreasing (for the order induced by $\overline \Omega$) for every $\zeta\in E$.
\end{cor}

\begin{proof}
	 This follows from Proposition~\ref{prop:2} and~\cite[Corollary 3.3]{CalziPeloso}.
\end{proof}

\begin{prop}\label{prop:1}
	Take $p_1,p_2,q_1,q_2\in (0,\infty]$, and $\vect{s_1},\vect{s_2}\in \R^r$. If $\Ac^{p_2,q_2}_{\vect{s_2}}(D)\neq \Set{0}$, then  $\Ac^{p_1,q_1}_{\vect{s_1},0}(D)\cap \Ac^{p_2,q_2}_{\vect{s_2}}(D)$ is dense in $\Ac^{p_1,q_1}_{\vect{s_1},0}(D)$ (resp.\ $\Ac^{p_1,q_1}_{\vect{s_1}}(D)\cap \Ac^{p_2,q_2}_{\vect{s_2}}(D)$ is dense in $\Ac^{p_1,q_1}_{\vect{s_1}}(D)$ for the weak topology $\sigma(\Ac^{p_1,q_1}_{\vect{s_1}}(D),\Lc^{p_1',q_1'}_{-\vect{s_1}+(1/q_1-1)_+\vect b+(1/p_1-1)_+\vect d}(D))$).
	Analogous assertions hold with $\Ac^{p_2,q_2}_{\vect{s_2},0}(D)$ in place of $\Ac^{p_2,q_2}_{\vect{s_2}}(D)$.
\end{prop}

\begin{proof}
	Argue as in~\cite[Proposition 3.9]{CalziPeloso}, using Corollary~\ref{oss:1}.
\end{proof}

\subsection{Reproducing Kernels}

Notice that $\Ac^{2,2}_{\vect s}(D)=A^{2,2}_{\vect s}(D)$, so that $P_{\vect b+\vect d-2\vect s}$ is the orthogonal projector of $\Lc^{2,2}_{\vect s}(D)$ onto $\Ac^{2,2}_{\vect s}(D)$.

\begin{prop}\label{prop:4}
	Take $p,q\in (0,\infty]$ and $\vect s,\vect s'\in \R^r$. Then, $B^{\vect s'}_{(\zeta,z)}\in \Ac^{p,q}_{\vect s,0}(D)$ (resp.\ $B^{\vect s'}_{(\zeta,z)}\in \Ac^{p,q}_{\vect s}(D)$) for some/every $(\zeta,z)\in D$ if and only if the following conditions hold:
	\begin{itemize}
		\item $\vect s \succ \frac{1}{2 q}\vect m $ (resp.\ $\vect s \Meg \vect 0$ if $q=\infty$);
		
		\item $\vect s'\prec \frac 1 p \vect d-\frac{1}{2p} \vect m'$ (resp.\ $\vect s'\meg \vect 0$ if $p=\infty$);
		
		\item $\vect s+\vect s'\prec \frac 1 q \vect b+\frac 1 p \vect d-\frac{1}{2 q}\vect m'$ (resp.\ $\vect s+\vect s'\meg\frac 1 p \vect d$ if $q=\infty$).
	\end{itemize}
	In this case,
	\[
	\norm{B^{\vect s'}_{(\zeta,z)}}_{\Ac^{p,q}_{\vect s}(D)}=\norm{B^{\vect s'}_{(0, i e_\Omega)}}_{\Ac^{p,q}_{\vect s}(D)} \Delta_\Omega^{\vect s+\vect s'-\vect b/q-\vect d/p}(\rho(\zeta,z))
	\]
	for every $(\zeta,z)\in D$.
\end{prop}

\begin{proof}
	This follows from~\cite[Lemmas 2.32, 2.35, and 2.39, and Corollaries 2.22 and 2.36]{CalziPeloso}.
\end{proof}

\begin{prop}\label{prop:3}
	Take $p,q\in (0,\infty]$ and $\vect s,\vect s'\in \R^r$ such that the following hold:
	\begin{itemize}
		\item $\vect s\succ \frac 1 q \vect b+\frac{1}{p}\vect d+\frac{1}{2 q'}\vect m'$;
		
		\item $\vect s'\prec \frac{1}{p'}\vect d-\frac{1}{2p'}\vect m', \vect b+\vect d-\frac 1 2 \vect m$;

		\item $\vect s+\vect s'\prec \frac{1}{\min(1,q)}\vect b+\frac{1}{\min(1,p)}\vect d-\frac{1}{2 q'}\vect m$ or $\vect s+\vect s'\meg \frac{1}{q}\vect b+\frac{1}{\min(1,p)}\vect d$  if $q'=\infty$.
	\end{itemize}
	Then, $P_{\vect{s'}}f=f$ for every $f\in \Ac^{p,q}_{\vect s}(D)$.
\end{prop}

\begin{proof}
	Argue as in the proof of~\cite[Proposition 3.13]{CalziPeloso}, using Propositions~\ref{prop:2} and~\ref{prop:4}.
\end{proof}

\subsection{Sampling}

\begin{deff}
	For every $\vect{s}\in \R^r$, we define $\Ms_{\vect s}$ as the space of $f\in \Hol(D)$ such that the function $(\zeta,z)\mapsto \Delta^{\vect s}_\Omega(\rho(\zeta,z)) \ee^{-\abs{\zeta}^{2 \alpha}-\abs{\Re z}^\alpha-\abs{\rho(\zeta,z)}^{\alpha}} f(\zeta,z)$ is bounded on $D$ for some $\alpha\in [0,1/2)$.
\end{deff}

Observe that $\Ms_{\vect s}\subseteq \Ms_{\vect s'}$ for $\vect s \meg \vect s'$, and that
\[
\Ac^{p,q}_{\vect s}(D)\subseteq \Ac^{\infty,\infty}_{\vect s-\vect b/q-\vect d/p}(D)\subseteq \Ms_{\vect s-\vect b/q-\vect d/p}
\]
for every $\vect s\in \R^r$, thanks to Proposition~\ref{prop:2}.

\begin{teo}\label{teo:1}
	Take $p,q\in (0,\infty]$, $\vect s\in \R^r$, $R_0>1$ and $\delta_+>0$, and $\vect s'\Meg \vect s-\frac 1 q\vect b-\frac 1 p\vect d$. Then, there are $\delta_-,C>0$ such that, for every $F$-$(\delta,R)$-lattice $(\zeta_k,z_{j,k})_{j\in J,k\in K}$ on $D$, with $R\in (1,R_0]$, if we define
	\[
	S_+\colon \Hol(D)\ni f \mapsto \bigg(\Delta_\Omega^{\vect s-\vect b/q-\vect d/p}(h_k) \max_{\overline B((\zeta_k,z_{j,k}),R\delta)} \abs{f}\bigg)\in \C^{J\times K}
	\]
	and
	\[
	S_-\colon \Hol(D)\ni f \mapsto \bigg(\Delta_\Omega^{\vect s-\vect b/q-\vect d/p}(h_k) \min_{\overline B((\zeta_k,z_{j,k}),R\delta)} \abs{f}\bigg)\in \C^{J\times K},
	\]
	where $h_k\coloneqq \rho(\zeta_k,z_{j,k})$ for every $j\in J$ and for every $k\in K$, then
	\[
	\frac{1}{C}\norm{f}_{\Ac^{p,q}_{\vect s}(D)}\meg \delta^{m/p+(2n+m)/q} \norm*{S_\eps f}_{\ell^{p,q}(J,K)}\meg C\norm{f}_{\Ac^{p,q}_{\vect s}(D)}
	\]
	for every $f\in \Hol(D)$ and for every $\delta\in (0,\delta_+]$ if $\eps=+$, and for every $f\in \Ms_{\vect s'}$ and for every $\delta\in (0,\delta_-]$ if $\eps=-$.
	In addition,
	\[
	\Ac^{p,q}_{\vect s,0}(D)=\Hol(D)\cap S_+^{-1}(\ell^{p,q}_0(J,K)) \qquad \text{and} \qquad\Ac^{p,q}_{\vect s,0}(D)=\Ms_{\vect s'}(D)\cap S_-^{-1}(\ell^{p,q}_0(J,K)),
	\]
	where the second equality holds provided that $\delta\meg \delta_-$.
\end{teo}

Before we pass to the proof, we need an analogue of~\cite[Lemma 3.26]{CalziPeloso}.

\begin{lem}\label{lem:1}
	There are $R'_0>0$ and a constant $C>0$ such that, for every $p,q\in (0,\infty]$, for every ${R'}\in (0,R'_0]$, for every $f\in \Hol(D)$ and for every $(\zeta,h)\in E\times\Omega$,
	\[
	\norm{f_h^{(\zeta)}}_{L^p(F)}\meg C^{1/\min(1,p,q)} \left( \dashint_{B_{E\times\Omega}((\zeta,h), {R'})}  \norm{f_{h'}^{(\zeta')}}_{L^p(F)}^q\,\dd \nu_{E\times\Omega}(\zeta,h)\right) ^{1/q}.
	\]
	(modification if $q=\infty$).
\end{lem}

\begin{proof}
	Set $\ell\coloneqq \min(1,p,q)$ to simplify the notation.
	By~\cite[Lemma 3.24]{CalziPeloso}, there are $R_0>0$ and $C'>0$ such that
	\[
	\abs{f(\zeta,z)}^{\ell}\meg C' \dashint_{B((\zeta,z),R)} \abs{f}^{\ell}\,\dd \nu_D
	\]
	for every $f\in \Hol(D)$, for every $(\zeta,z)\in D$, and for every $R\in (0,R_0]$.  Then, applying Minkowski's integral inequality (with exponent $\frac{p}{\ell}$) and Young's inequality,
	\[
	\begin{split}
		\norm{f_h^{(\zeta)}}_{L^p(F)}^\ell&\meg C' C'_R  \dashint_{B_{E\times\Omega}((\zeta,h),R)} \norm*{\abs{f_{h'}^{(\zeta')}}^{\ell}*[(\chi_{B((\zeta, i\Phi(\zeta)+i h),R)})_{h'}^{(\zeta')}]\check{\;}  }_{L^{p/\ell}(F)}  \Delta_\Omega^{\vect d}(h') \,\dd \nu_{E\times\Omega}(\zeta',h')\\
		&\meg C''  \dashint_{B_{E\times\Omega}((\zeta,h),R)} \norm*{f^{(\zeta')}_{h'}}_{L^{p}(F)}^\ell \frac{\Delta_\Omega^{\vect d}(h') }{\Delta_\Omega^{\vect d}(h) } \,\dd \nu_{E\times\Omega}(\zeta',h')
	\end{split}
	\]
	for every $f\in \Hol(D)$ and for every $h\in \Omega$,
	where 
	\[
	C'_R\coloneqq \frac{\nu_{E\times\Omega}(B_{E\times\Omega}((0,e_\Omega),R))}{\nu_D(B((0, i e_\Omega),R))}
	\]
	and 
	\[
	C''\coloneqq C' \sup\limits_{0<R\meg R_0} \sup\limits_{(\zeta',h')\in E\times\Omega} C'_R \norm*{(\chi_{B((0,i e_{\Omega}),R)})_{h'}^{(\zeta')}}_{L^1(F)}.
	\]
	By~\cite[Corollary 2.49]{CalziPeloso}, there is a constant $C>0$ such that, for every $f\in \Hol(D)$ and for every $h\in \Omega$,
	\[
	\begin{split}
		\norm{f_h^{(\zeta)}}_{L^p(F)}^\ell&\meg C\dashint_{B_{E\times\Omega}((\zeta,h),R)} \norm*{f^{(\zeta')}_{h'}}_{L^{p}(F)}^\ell  \,\dd \nu_{E\times\Omega}(\zeta',h'),
	\end{split}
	\]
	if $R\in (0,R_0]$. Then, Jensen's inequality (with exponent $ \frac{q}{\ell}$) leads to the conclusion.
\end{proof}

\begin{proof}[Proof of Theorem~\ref{teo:1}.]
	For the sake of simplicity, we shall generally present the computations as if $p,q<\infty$. We leave to the reader  the (purely formal) modifications which are necessary when $\max(p,q)=\infty$. Throughout the proof, for every $t\in T_+$ we shall denote by $g_t$ an element of $GL(E)$ such that $t\cdot \Phi=\Phi\circ(g_t\times g_t)$.
	
	\textsc{Step I.} Define, for every $R'>0$, for every $\zeta\in E$, and for every $h\in \Omega$, 
	\[
	M_{R'}(\zeta,h)\coloneqq \norm*{\left(\chi_{B((0,i e_\Omega),  {R'})}\right)_h^{(\zeta)}}_{L^1(F)}\meg \Hc^m(\pr_F(B((0,i e_\Omega),  {R'})))<\infty
	\]
	Then, for every $\ell\in (0,\infty]$, for every $t'\in T_+ $, and for every $(\zeta',x')\in \Nc$,
	\[
	\norm*{\left(\chi_{B((\zeta',x'+i\Phi(\zeta)+i t'\cdot e_\Omega), {R'})}\right)^{(\zeta)}_h}_{L^\ell(F)}=\Delta^{-\vect d/\ell}(t') M_{R'}(g_{t'}^{-1}(\zeta-\zeta'),t'^{-1}\cdot h)^{1/\ell},
	\]
	with the convention $0^0=0$. 
	In particular, 
	\[
	\norm*{\left(\chi_{B((\zeta_{k},z_{j,k}), {R'})}\right)^{(\zeta)}_h}_{L^\ell(F)}= \chi_{B_{E\times\Omega}((\zeta_k,h_k),{R'})}(\zeta,h) 
	\Delta_\Omega^{-\vect d/\ell}(h_k) M_{R'}(g_{t_k}^{-1}(\zeta-\zeta_k), t_k^{-1}\cdot h)^{1/\ell}
	\]
	for every $(\zeta,h)\in E\times\Omega$ and for every $k\in K$, where $t_k\in T_+$ is such that $h_k=t_k\cdot e_\Omega$. 
	In addition,  
	\[
	\norm{M_{R'}}_{L^\infty(\nu_{E\times\Omega})}\asymp {R'}^{m} \text{ for ${R'}\to 0^+$}.
	\]
	For every $(\zeta,h)\in E\times\Omega$, define 
	\[
	K_{\zeta,h}\coloneqq \Set{k\in K\colon (\zeta,h)\in B_{E\times\Omega}((\zeta_k,h_k),R\delta)},
	\]
	and observe that there is $N\in\N$ such that $ \card(K_{\zeta,h})\meg N$ for every $(\zeta,h)\in E\times\Omega$, provided that $R\meg R_0$ and $\delta\meg \delta_+$. We may also assume that every $(\zeta,h)\in E\times\Omega$ is contained in at most $N$ balls $B_{E\times \Omega}((\zeta_k,h_k),2 R\delta)  $, $k\in K$, and that every $(\zeta,z)\in D$ is contained in at most $N$ balls $B((\zeta_{k},z_{j,k}),2 R\delta)$, $(j,k)\in J\times K$, provided that $R\meg R_0$ and $\delta\meg \delta_+$. 
	Finally, set $\ell\coloneqq \min (1,p,q)$.
	
	\textsc{Step II.} Let us prove that $S_+$ maps $ \Ac^{p,q}_{\vect s}(D)$  into $\ell^{p,q}(J,K)$. Take $f\in \Ac^{p,q}_{\vect s}(D)$ and define 
	\[
	C_{D,{R'}}\coloneqq \nu_D(B((0,i e_\Omega),{R'})), \quad C_{E\times \Omega,{R'}}\coloneqq \nu_{E\times \Omega}(B_{E\times \Omega}((0,e_\Omega),{R'})), \quad \text{and} \quad C_{\Omega,{R'}}\coloneqq \nu_\Omega(B_\Omega(e_\Omega,{R'}))
	\] 
	for every ${R'}>0$ to simplify the notation.
	Then,~\cite[Lemma 3.24]{CalziPeloso} implies that there are $R'_0\in (0,1/2]$ and $C_1>0$ such that
	\[
	\max_{\overline B((\zeta_{k},z_{j,k}),R\delta)} \abs{f}^p\meg  \frac{C_1}{ C_{D,R'_0 \delta}} \int_{B((\zeta_{k},z_{j,k}), (R+R'_0)\delta)} \abs{f}^p\,\dd \nu_D
	\]
	for every $(j,k)\in J\times K$.
	Therefore,~\cite[Corollary 2.49]{CalziPeloso} implies that there is a constant $C_{2}>0$ such that
	\[
	(S_+ f)_{j,k}^p \meg \frac{C_{2}}{C_{D,R'_0\delta}}\Delta_\Omega^{p\vect s-(p/q)\vect b}(h_k) \int_{E\times\Omega} \int_F \abs{(\chi_{B((\zeta_{k},z_{j,k}),(R+R'_0)\delta)}f)^{(\zeta)}_h(x)}^p\,\dd x\,\dd \nu_{E\times\Omega}(\zeta, h)
	\]
	for every $(j,k)\in J\times K$.
	Hence, 
	\[
	\sum_{j\in J}(S_+ f)_{j,k}^p \meg \frac{C_{2}}{C_{D,R'_0\delta}} N\Delta_\Omega^{p\vect s-(p/q)\vect b}(h_k) \int_{B_{E\times\Omega}((\zeta_k,h_k), (R+R'_0)\delta)} \norm{f_h^{(\zeta)}}_{L^p(F)}^p\,\dd \nu_{E\times\Omega}(\zeta,h)
	\]
	for every $k\in K$.
	Now,  Lemma~\ref{lem:1} shows that there is a constant $C_{3}>0$ such that 
	\[
	\begin{split}
		&\int_{B_{E\times\Omega}((\zeta_k,h_k), (R+R'_0)\delta)} \norm{f_h^{(\zeta)}}_{L^p(F)}^p\,\dd \nu_{E\times\Omega}(\zeta,h)\\
		&\qquad\meg C_3 \int_{B_{E\times\Omega}((\zeta_k,h_k), (R+R'_0)\delta)}\left(  \dashint_{B_{E\times\Omega}((\zeta',h'), R'_0\delta)}  \norm{f_{h}^{(\zeta)}}_{L^p(F)}^q\,\dd \nu_{E\times\Omega}(\zeta,h)\right) ^{p/q}\,\dd\nu_{E\times\Omega}(\zeta',h')\\
		&\qquad \meg C_3 \frac{C_{E\times\Omega, (R+R'_0)\delta}}{C_{E\times\Omega,R'_0\delta}^{p/q}  } \left(  \int_{B_{E\times\Omega}((\zeta_k,h_k), (R+2R'_0)\delta)}  \norm{f_{h}^{(\zeta)}}_{L^p(F)}^q\,\dd \nu_{E\times\Omega}(\zeta,h)\right) ^{p/q}
	\end{split}
	\]
	for every $k\in K$.
	Therefore, another application of~\cite[Corollary 2.49]{CalziPeloso}  shows that there is a constant $C_2'>0$ such that
	\[
	\norm{S_+ f}_{\ell^{p,q}(J,K)}\meg \frac{C_2' C_{E\times\Omega,(R+R'_0)\delta}^{1/p} }{C_{D,R'_0\delta}^{1/p}C_{E\times\Omega, R'_0\delta}^{1/q} } N^{1/p+1/q} \norm{f}_{\Ac^{p,q}_{\vect s}(D)}.
	\]
	
	Next, let us prove that $S_+(\Ac^{p,q}_{\vect s,0}(D))\subseteq \ell^{p,q}_0(J,K)$. Observe that we may assume that $\vect{s}\succ \frac{1}{2 q}\vect{m}$. Then, take $\tilde p\in (0,p)$ and $\tilde q\in (0,q)$ so that $ \vect{s''}\coloneqq \vect s-\big(\frac 1 q-\frac{1}{\tilde q}\big)\vect b-\big(\frac 1 p-\frac{1}{\tilde p}\big)\vect d\succ\frac{1}{2 \tilde q}\vect{m}$, and observe that the preceding computations show that $S_+(\Ac^{p,q}_{\vect s,0}(D)\cap \Ac^{\tilde p,\tilde q}_{\vect s''}(D))\subseteq  \ell^{\tilde p,\tilde q}(J,K)\subseteq \ell^{p,q}_0(J,K)$, so that the assertion follows by means of  Proposition~\ref{prop:1}.
	
	\textsc{Step III.}  Now, take $f\in \Hol(D)$ and assume that $S_+ f\in \ell^{p,q}_0(J,K)$ (resp.\ $S_+ f\in \ell^{p,q}(J,K)$), and let us prove that $f\in \Ac^{p,q}_{\vect s,0}(D)$ (resp.\ $f\in \Ac^{p,q}_{\vect s}(D)$).
	Observe first that
	\[
	\abs{f_h^{(\zeta)}}\meg \sum_{(j,k)\in J\times K_{\zeta,h}} \Delta_\Omega^{\vect b/q+\vect d/p-\vect s}(h_k) \left( \chi_{B((\zeta_{k},z_{j,k}),R\delta)} \right)_h^{(\zeta)} (S_+ f)_{j,k}
	\]
	on $F$, for every $(\zeta,h)\in E\times \Omega$, so that $f_h^{(\zeta)}\in L^p_0(F)$ (resp.\ $f_h^{(\zeta)}\in L^p(F)$) for every $(\zeta,h)\in E\times \Omega$. In addition,
	\[
	\begin{split}
		\norm{f_h^{(\zeta)}}_{L^p(F)}&\meg N^{1/p'} \norm*{\Big(\Delta_\Omega^{\vect b/q-\vect s}(h_k) M_{R\delta}(g_{t_k}^{-1}(\zeta-\zeta_k),t_k^{-1}\cdot h )^{1/p}(S_+ f)_{j,k}\Big)_{j,k}}_{\ell^p(J\times K_{\zeta,h})}    \\
		&\meg   N^{1/p'+(1/p-1/q)_+}\norm{M_{R\delta}}_\infty^{1/p} \norm*{ \left( \Delta_\Omega^{\vect b/q-\vect s}(h_k)\norm*{  ((S_+ f)_{j,k})_j}_{\ell^p(J)}\right) _k}_{\ell^q(K_{\zeta,h})}  
	\end{split}
	\]
	for every $(\zeta,h)\in E\times \Omega$, so that $(\zeta,h)\mapsto \norm{f_h^{(\zeta)}}_{L^p(F)} $ belongs to $L^q_0(E\times \Omega)$ and there is a constant $C_2''>0$ (cf.~\cite[Corollary 2.49]{CalziPeloso}) such that
	\[
	\norm{f}_{\Ac^{p,q}_{\vect s}(D)}\meg N^{1/p'+\max(1/p,1/q)}\norm{M_{R\delta}}_\infty^{1/p}  C_{E\times \Omega,R\delta}^{1/q} C_2'' \norm{S_+ f}_{\ell^{p,q}(J,K)}.
	\]

	\textsc{Step IV.} Observe that, if $(\zeta'_{k'},z'_{j',k'})_{j'\in J',k'\in K'}$ is an $F$-$(\delta,R')$-lattice on $D$, then there are two mappings $\iota_1\colon K'\to K$ and $\iota_2\colon J'\times K' \to J$ such that, setting $h'_{k'}\coloneqq \rho(\zeta'_{k'},z'_{j',k'})$ for every $j'\in J'$ and for every $k'\in K'$, 
	\[
	(\zeta'_{k'},h'_{k'})\in B_{E\times \Omega}((\zeta_{\iota_1(k')},h_{\iota_1(k')}),R\delta)
	\]
	and
	\[
	(\zeta'_{k'},z'_{j',k'})\in B((\zeta_{\iota_1(k')},z_{\iota_2(j',k'),\iota_1(k')}),R\delta)
	\]
	for every $j'\in J'$ and for every $k'\in K'$. Define
	\[
	S_-'\colon \Hol(D)\ni f \mapsto \bigg(\Delta_\Omega^{\vect s-\vect b/q-\vect d/p}(h'_{k'}) \min_{\overline B((\zeta'_{k'},z'_{j',k'}),(R+R')\delta)} \abs{f}\bigg)\in \C^{J'\times K'},
	\]
	and observe that, by~\cite[Corollary 2.49]{CalziPeloso} and the preceding remarks, there is a constant $C'>0$ such that
	\[
	(S'_- f)_{j',k'}\meg C' (S_- f)_{\iota_2(j',k'),\iota_1(k')}
	\]
	for every $f\in \Hol(D)$, for every $j'\in J'$ and for every $k'\in K'$. In addition, there is $N'\in \N$ such that the fibres of $\iota_1$ and $(j',k')\mapsto (\iota_2(j',k'),\iota_1(k'))$  have at most $N'$ elements. Consequently, 
	\[
	\norm{S'_- f}_{\ell^{p,q}(J',K')}\meg C' N'^{1/p+1/q} \norm{S_- f}_{\ell^{p,q}(J,K)}
	\]
	for every $f\in \Hol(D)$. In addition, if $S_- f\in \ell^{p,q}_0(J,K)$, then $S_-' f\in \ell_0^{p,q}(J',K')$. 
	
	Observe, furthermore, that if $R'\Meg 8$, then we may choose $(\zeta'_{k'},z'_{j',k'})_{j'\in J',k'\in K'}$ such that $K'=K'_1\times K'_2$, such that $h'_{(k'_1,k'_2)}=h'_{(k''_1,k'_2)}$ for every $k'_1,k''_1\in K'_1$ and for every $k_2'\in K'_2$, and such that $(h'_{(k_1',k'_2)})_{k'_2\in K'_2}$ is a $(\delta,R')$-lattice on $\Omega$ for some/every $k'_1\in K_1'$ (argue as in the proof of~\cite[Lemma 2.55]{CalziPeloso}).
	
	\textsc{Step V.} Take  $f\in \Ms_{\vect s'}(D)$ such that $S_- f\in \ell^{p,q}(J,K)$ and let us prove that
	\[
	\norm{f}_{\Ac^{p,q}_{\vect s}(D)}\meg C \delta^{m/p+(2 n+m)/q} \norm{S_- f}_{\ell^{p,q}(J,K)}
	\] 
	for a suitable constant $C>0$ (depending only on $\delta_-$ and $R_0$), provided that $\delta_-$ is sufficiently small. 
	Observe first that, by~\textsc{step IV}, up to replacing $R$ with $R+8$, we may assume that $(\zeta_{k},z_{j,k})$ is an $F$-$(\delta,R)$-lattice, that $K=K'\times K''$, that $h_{(k',k'')}$ only depends on $k''$ (so that we also write $h_{k''}$ instead of $h_{(k',k'')}$), and that $(h_{k''})$ is a $(\delta,R)$-lattice in $\Omega$.	
	Observe that, for every $(j,k)\in J\times K$, we may find $(\zeta'_{j,k},z'_{j,k})\in \overline B((\zeta_{k},z_{j,k}),R\delta)$ such that
	\[
	\abs{f(\zeta'_{j,k},z'_{j,k})}=\min_{\overline B((\zeta_{k},z_{j,k}),R\delta)}\abs{f}.
	\]
	Now,~\cite[Lemmas 3.24 and 3.25]{CalziPeloso}  imply that there are $R'_1\in (0,R_0']$ and  $C_3>0$ such that, for every $j\in J$, for every $k\in K$, for every $(\zeta,x)\in \Nc$, and for every $h\in \Omega$ such that $d((\zeta_{k},z_{j,k}),(\zeta,x+i\Phi(\zeta)+i h) )<R\delta$,
	\[
	\abs{f_h^{(\zeta)}(x)}\meg\abs{f(\zeta'_{j,k},z'_{j,k})}+C_{3} R \delta \norm{\chi_{B((\zeta,x+i\Phi(\zeta)+i h), 2 R \delta+R'_1)} f }_{L^p(\nu_D)},
	\]
	provided that $R\delta \meg R'_1$.
	Then,
	\[
	\begin{split}
		\norm{f_h^{(\zeta)}}_{L^p(F)}^p&\meg 2^{(p-1)_+}\norm{M_{R\delta}}_{L^\infty(E\times\Omega)} \sum_{(j,k)\in J\times K_{\zeta,h}}\Delta_\Omega^{-\vect d}(h_k) \abs{f(\zeta'_{j,k},z'_{j,k}) }^p\\
		&\qquad+2^{(p-1)_+} (C_3 R \delta)^p \Theta_1(\zeta,h)  ,
	\end{split}
	\]
	where  
	\[
	\begin{split}
		\Theta_1(\zeta,h)\coloneqq\sum_{(j,k)\in J\times K_{\zeta,h}} \int_F (\chi_{B((\zeta_{k},z_{j,k}),R\delta)})_h^{(\zeta)}(x)\int_D \chi_{B((\zeta,x+i\Phi(\zeta)+i h),2 R\delta+R'_1)} \abs{f }^p\,\dd \nu_D \,\dd x.
	\end{split}
	\]
	Now, set 
	\[
	K''_{h}\coloneqq \Set{k''\in K''\colon h\in B_\Omega(h_{k''},R\delta)  }
	\]
	for every $h\in \Omega$, and observe that we may assume that, for every $h\in \Omega$,
	\[
	\card K''_h\meg N,
	\] 
	provided that $\delta\meg \delta_+$, as in~\textsc{step I}. 
	Then,~\cite[Corollary 2.49]{CalziPeloso} implies that there is a constant $C_4>0$ such that, if $R\delta \meg R'_1$,
	\[
	\begin{split}
		& \int_{E}  \left(\sum_{(j,k)\in J\times K_{\zeta,h}} \Delta_\Omega^{-\vect d}(h_k) \abs{f(\zeta'_{j,k},z'_{j,k}) }^p\right)^{q/p}\,\dd\zeta\\
		&\qquad \meg N^{(q/p-1)_+} \sum_{k''\in K''_h} \Delta_\Omega^{-(q/p)\vect d}(h_k)\int_{E} \sum_{k'\colon (k',k'')\in K_{\zeta,h}}  \left(\sum_{j\in J}  \abs{f(\zeta'_{j,(k',k'')},z'_{j,(k',k'')}) }^p\right)^{q/p}\,\dd \zeta \\
		&\qquad\meg C_{E,R\delta}  N^{\max(1,q/p)} \sum_{k''\in K''_h} \Delta_\Omega^{-q(\vect b/q+ \vect d/p)}(h_{k''})\sum_{k'\in K'}  \left(\sum_{j\in J}  \abs{f(\zeta'_{j,(k',k'')},z'_{j,(k',k'')}) }^p\right)^{q/p} \\
		&\qquad \meg C_4 \delta^{2n} N^{\max(1,q/p)}\Delta_\Omega^{-q \vect s}(h)  \sum_{k''\in K''_h}\sum_{k'\in K'}  \left(\sum_{j\in J}  \abs{(S_- f)_{j,(k',k'')} }^p\right)^{q/p}
	\end{split}
	\]
	where the first inequality follows from the convexity or subadditivity of the mapping $x\mapsto x^{q/p}$ on $\R_+$, while the second one follows from Tonelli's theorem, setting $C_{E,R\delta}\coloneqq \Hc^{2n}(\pr_E(B_{E\times \Omega}((0,e_\Omega),R\delta)))$.
	
	Now, observe that
	\[
	\begin{split}
		&\Theta_1(\zeta,h)=\int_D  \abs{f(\zeta',z')}^p \Theta_2(\zeta',z',\zeta,h)  \,\dd \nu_D(\zeta',z'),
	\end{split}
	\]
	where
	\[
	\Theta_2(\zeta',z',\zeta,h)\coloneqq\int_F\sum_{(j,k)\in J\times K_{\zeta,h}} (\chi_{B((\zeta_{k},z_{j,k}),R\delta)\cap B((\zeta',z'),2 R\delta+R'_1)})^{(\zeta)}_h(x)   \,\dd x .
	\]
	In addition, for every $(\zeta',z')\in D$ and for every $(\zeta,h)\in E\times\Omega$, setting $h'\coloneqq \rho(\zeta',z')$, one has
	\[
	\begin{split}
		\Theta_2(\zeta',z',\zeta,h)&\meg N \norm{(\chi_{B((\zeta',z'),2 R\delta+R'_1)})_h^{(\zeta)}  }_{L^1(F)}=N M_{2 R\delta+R'_1}(g_{t'}^{-1}(\zeta-\zeta'),t'^{-1}\cdot h) \Delta_\Omega^{-\vect d}(h'),
	\end{split}
	\]
	provided that $R\meg R_0$ and $\delta\meg \delta_+$, where $t'\in T_+$ is such that $h'=t'\cdot e_\Omega$.
	Therefore, by~\textsc{step I} we see that
	\[
	\Theta_1(\zeta,h) \meg N \norm{M_{2R\delta+R'_1}}_{L^\infty(E\times \Omega)}\int_{B_{E\times\Omega}((\zeta,h),2 R\delta+R'_1)} \norm{f_{h'}^{(\zeta')}}_{L^p(F)}^p\,\dd \nu_{E\times\Omega}(\zeta',h'),
	\]
	provided that $R\meg R_0$ and $\delta\meg \delta_+$.
	Then, applying Lemma~\ref{lem:1} as in~\textsc{step II}, we see that there is a constant $C_5>0$ such that
	\[
	\int_{E} \Theta_1(\zeta,h)^{q/p} \,\dd\zeta\meg C_5 \int_{ B_\Omega(h,2 R\delta+2 R'_1) } \norm{f_{h'}}_{L^{p,q}(F,E)}^q\,\dd \nu_\Omega(h')
	\]
	provided that $R\meg R_0$ and $\delta\meg \delta_+$.
	
	Therefore, there is a constant $C_6>0$ such that
	\[
	\norm{f_h}_{L^{p,q}(F,E)}\meg C_6 \Delta^{-\vect s}_\Omega(h)\delta^{m/p+2n/q} \norm{S_- f}_{\ell^{p,q}(J,K'\times K''_h)}+\delta C_6\left(  \int_{ B_\Omega(h,2 R\delta+2 R'_1) } \norm{f_{h'}}^q_{L^{p,q}(F,E)}\,\dd \nu_\Omega(h')\right)^{1/q} ,
	\]
	provided that $R\meg R_0$ and $\delta\meg \min(R_1'/R_0,\delta_+)$.
	
	Now, by assumption, there is $\alpha'\in (0,1/2)$ such that
	\[
	\sup_{(\zeta,z)\in D} \Delta_\Omega^{\vect{s'}}(\rho(\zeta,z)) \ee^{-\abs{\zeta}^{2\alpha'}-\abs{\Re z}^{\alpha'}-\abs{\rho(\zeta,z)}^{\alpha'}} \abs{f(\zeta,z)}<\infty.
	\]
	Then, take $(g^{(\eps)})_{\eps>0}$ as in~\cite[Lemma 1.22]{CalziPeloso} for some $\alpha\in (\alpha',1/2)$, so that $G(\eps)\coloneqq f g^{(\eps)}\in \Ac^{p,\infty}_{\vect{s'}}(D)$ and  $S_- G(\eps)\meg S_- f$ for every $\eps>0$. In particular,  the mapping $h\mapsto \norm{G(\eps)_h}_{L^{p,q}(F,E)}$ is (finite and) decreasing on $\Omega$, thanks to Corollary~\ref{oss:1}, for every $\eps>0$.
	In addition, observe that we may take $\delta_1\in (0,\min(R'_1/R_0,\delta_+)]$ and $R'_1$ so small that $B_\Omega(e_\Omega, 2 R_0\delta_1+2R'_1)\subseteq e_\Omega/2+\Omega$. Then, by homogeneity,
	\[
	B_\Omega(h, 2 R_0\delta_1+2R'_1)\subseteq h/2+\Omega
	\]
	for every $h\in \Omega$. Then, the preceding estimates (applied to $G(\eps)$) show that there is a constant $C_6'>0$ such that
	\[
	\begin{split}
		\norm*{G(\eps)_h}_{L^{p,q}(F,E)}&\meg C_6' \Delta_\Omega^{-\vect s}(h)\delta^{m/p+2n/q} \norm{S_- f}_{\ell^{p,q}(J, K'\times K''_{h})}+\delta C_6' 	\norm*{G(\eps)_{h/2}}_{L^{p,q}(F,E)},
	\end{split}
	\]
	for every $\eps>0$, provided that $\delta\meg \delta_1$ and $R\meg R_0$.
	If we define 
	\[
	\chi_\ell\colon D\ni (\zeta,z)\mapsto \chi_{e_\Omega/2^\ell+\Omega }(\rho( \zeta,z))\in \R_+
	\]
	for every $\ell\in \N$, then there is a constant $C_6''>0$ such that
	\[
	\norm*{\chi_\ell G(\eps)}_{\Lc^{p,q}_{\vect s}(D)}\meg C_6'' \delta^{m/p+(2n+m)/q}\norm{S_- f}_{\ell^{p,q}(J,K)}+\delta C_6'' \norm*{\chi_{\ell+1}G(\eps)}_{\Lc^{p,q}_{\vect s}(D)}
	\]
	for every $\eps>0$ and for every $\ell\in \N$, provided that $\delta\meg \delta_1$ and $R\meg R_0$. 
	Now, define 
	\[
	e_{\alpha'}\colon D\ni (\zeta,z)\mapsto \ee^{\abs{\zeta}^{2\alpha'}+\abs{\Re z}^{\alpha'}+\abs{\Im z-\Phi(\zeta)}^{\alpha'}}\in \R_+,
	\]
	and observe that
	\[
	\begin{split}
		\norm*{G(\eps) \chi_{\ell}}_{\Lc^{p,q}_{\vect s}(D)}&\meg \norm{e_{\alpha'}^{-1} f}_{\Lc^{\infty,\infty}_{\vect{s'}}(D)} \norm*{e_{\alpha'} g^{(\eps)}\chi_{\ell}}_{\Lc^{p,q}_{\vect s-\vect{s'}}(D)}\\
		&\meg C_{7,\eps} \norm{e_{\alpha'}^{-1} f}_{\Lc^{\infty,\infty}_{\vect{s'}}(D)} \norm*{\chi_{e_\Omega/2^{\ell}+\Omega} \Delta_\Omega^{\vect s-\vect{s'}}\ee^{- C_{7,\eps}\abs{\,\cdot\,}^\alpha}}_{L^q(\nu_\Omega)}\\
		& \meg C_{7,\eps}' \norm{e_{\alpha'}^{-1} f}_{\Lc^{\infty,\infty}_{\vect{s'}}(D)} \norm{\chi_{e_\Omega/2^{\ell}+\Omega} \Delta_\Omega^{\vect s-\vect{s'}+\vect d/q}}_{L^\infty(\Omega)} \norm{\ee^{- C_{7,\eps}\abs{\,\cdot\,}^\alpha}}_{L^q(\Omega)}\\
		&\meg  C_{7,\eps}'' 2^{(\vect{s'}-\vect{s}-\vect d/q)\ell}\norm{e_{\alpha'}^{-1} f}_{\Lc^{\infty,\infty}_{\vect{s'}}(D)}
	\end{split}
	\]
	for suitable constants $C_{7,\eps},C_{7,\eps}',C_{7,\eps}''>0$, since $\vect s-\vect{s'}+\vect d/q\meg \vect 0$ (cf.~\cite[Corollary 2.36]{CalziPeloso}). Then, fix $N'> \sum_{j=1}^r(s'_j-s_j)+m/q$ and choose $\delta_-\in (0,\delta_1]$ so that $ C_6''\delta_- \meg 2^{-N'}$. Observe that the preceding computations show that, if $\delta\in (0,\delta_-]$, then
	\[
	\begin{split}
		\norm*{G(\eps)\chi_\ell}_{\Lc^{p,q}_{\vect s}(D)}&=\sum_{\ell'\in\N} 2^{-\ell' N'}\left(\norm*{G(\eps)\chi_{\ell+\ell'}}_{\Lc^{p,q}_{\vect s}(D)}-\frac{1}{2^{N'}} \norm*{G(\eps)\chi_{\ell+\ell'+1}}_{\Lc^{p,q}_{\vect s}(D)}\right)\\
		&\meg \frac{C_6''}{1-2^{-N'}}\delta^{m/p+(2n +m)/q}\norm{S_- f}_{\ell^{p,q}(J,K)}
	\end{split}
	\]
	for every $\eps>0$ and for every $\ell\in \N$. Passing to the limit for $\ell\to \infty$, we then infer that $G(\eps)\in \Ac^{p,q}_{\vect s}(D)$ and that
	\[
	\norm*{G(\eps)}_{\Ac^{p,q}_{\vect s}(D)}\meg \frac{C_6''}{1-2^{-N'}}\delta^{m/p+(2n +m)/q}\norm{S_- f}_{\ell^{p,q}(J,K)}
	\]
	for every $\eps>0$.
	Then, passing to the limit for $\eps\to 0^+$, we infer that $f\in \Ac^{p,q}_{\vect s}(D)$ and that
	\[
	\norm{f}_{\Ac^{p,q}_{\vect s}(D)}\meg \frac{C_6''}{1-2^{-N'}}\delta^{m/p+(2n +m)/q}\norm{S_- f}_{\ell^{p,q}(J,K)}.
	\]
	
	\textsc{Step VI.} It only remains to prove that $f\in \Ac^{p,q}_{\vect s,0}(D)$ for every $f\in \Ac^{p,q}_{\vect s}(D)$ such that $S_- f\in \ell^{p,q}_0(J,K)$, provided that $\delta_-$ is sufficiently small.
	Observe first that the preceding computations show that
	\[
	\norm{\chi_\ell (f-G(\eps))}_{\Lc^{p,q}_{\vect s}(D)}\meg  \frac{C_6''}{1-2^{-N'}} \delta^{m/p+(2 n+m)/p} \norm{S_-(f-G(\eps))}_{\ell^{p,q}(J,K)}
	\]
	for every $\ell\in\N$ and for every $\eps>0$.
	Since $S_-(f-G(\eps))\meg (S_- f) \widetilde S_+(1-g^{(\eps)})$, where 
	\[
	\big[\widetilde S_+(1-g^{(\eps)})\big]_{j,k}\coloneqq\max_{B((\zeta_{k},z_{j,k}),R\delta)} \abs*{1-g^{(\eps)}}
	\]
	for every $(j,k)\in J\times K$, and since $1-g^{(\eps)}\to 0$ locally uniformly, it is readily seen that $\chi_\ell (f-G(\eps))\to 0$ in $\Lc^{p,q}_{\vect s}(D)$ for $\eps\to 0^+$, for every $\ell\in\N$. 
	In particular, $f_h\in L^{p,q}_0(F,E)$ for every $h\in \Omega$, and the mapping $h\mapsto \chi_{e_\Omega/2^\ell+\Omega}(h)\Delta^{\vect s}_\Omega(h)\norm{f_h}_{L^{p,q}(F,E)} $ belongs to $L^q_0(\nu_\Omega)$ for every $\ell\in \N$. To conclude, it will essentially suffice to show that, if $q=\infty$, then $\Delta^{\vect s}_\Omega(h)\norm{f_h}_{L^{p,\infty}(F,E)}\to 0$ as $h$ approaches the boundary of $\Omega$.
	Observe that, by the preceding computations, there is a constant $C_8>0$ (namely, $C_6'\max(1,\delta_-^{m/p})$) such that
	\[
	\norm{f_h}_{L^{p,\infty}(F,E)}\meg C_8 \Delta_\Omega^{-\vect s}(h)\norm{S_- f}_{\ell^{p,\infty}(J, K'\times K''_h)}+\delta C_8 \norm{f_{h/2}}_{L^{p,\infty}(F,E)}
	\]
	for every $h\in \Omega$. Observe that 
	\[
	\Delta_\Omega^{\vect s}(h) \norm{f_{h/2}}_{L^{p,\infty}(F,E)}= 2^{\vect s} \Delta_\Omega^{\vect s}(h/2) \norm{f_{h/2}}_{L^{p,\infty}(F,E)}\meg  2^{\vect s}  \norm{f}_{\Ac^{p,\infty}_{\vect s}(D)}
	\] 
	for every $h\in \Omega$. Therefore, assuming that $\delta_-$ is so small that $\delta_- C_8 2^{\vect s}\meg 1/2 $,
	\[
	\begin{split}
		\Delta_\Omega^{\vect s}(h)\norm{f_h}_{L^{p,\infty}(F,E)}&= \sum_{\ell\in \N} 2^{-\ell}\left(\Delta_\Omega^{\vect s}(h/2^{\ell}) \norm{f_{h/2^{\ell}}}_{L^{p,\infty}(F,E)}-\frac{1}{2 \cdot 2^{\vect s}}\Delta_\Omega^{\vect s}(h/2^{\ell}) \norm{f_{h/2^{\ell+1}}}_{L^{p,\infty}(F,E)}\right)\\
		&\meg C_8  \sum_{\ell\in \N} 2^{-\ell}\norm{S_- f}_{\ell^{p,\infty}(J, K'\times K''_{h/2^\ell})}
	\end{split}
	\]
	for every $h\in \Omega$. Now, observe that $\eta\coloneqq \min_{\ell\in \N} d_\Omega(e_\Omega,e_\Omega/2^{\ell+1})>0$, and that $\eta=\min_{\ell\in \N} d_\Omega(h,h/2^{\ell+1})$ for every $h\in\Omega$, by homogeneity. Therefore, if $\delta_-$ is so small that $2 R_0\delta_-<\eta$, then the sets $K''_{h/2^\ell}$, as $\ell$ runs through $\N$, are pairwise disjoint for every $h\in \Omega$. Hence,
	\[
	\Delta_\Omega^{\vect s}(h)\norm{f_h}_{L^{p,\infty}(F,E)}\meg 2 C_8  \norm{S_- f}_{\ell^{p,\infty}(J, K'\times K'''_h)}
	\]
	for every $h\in\Omega$, where $K'''_h\coloneqq \bigcup_{\ell\in \N} K''_{h/2^\ell}$. Since $K'''_h$ is contained in the complement of every \emph{fixed} finite subset of $K$ if $h\in\Omega\setminus (e_\Omega/2^\ell+\Omega)$ and $\ell$ is sufficiently large, this and the preceding remarks prove that $\Delta_\Omega^{\vect s}(h) \norm{f_h}_{L^p(\Nc)}\to 0$ as $h\to \infty$ in $\Omega$, provided that $\delta_-$ is sufficiently small (independently of $f$). The proof is complete. 
\end{proof}

\subsection{Atomic Decomposition and Duality}\label{sec:atomic}

\begin{deff}\label{38}
	Take $p,q\in (0,\infty]$ and $\vect s,\vect{s'}\in \R^r$.
	Then, we say that property $(\Lc)^{p,q}_{\vect s,\vect{s'},0}$ (resp.\ $(\Lc)^{p,q}_{\vect s,\vect{s'}}$) holds if for every $\delta_0>0$ there is an $F$-$(\delta,4)$-lattice $(\zeta_{k},z_{j,k})_{j\in J, k\in K}$, with $\delta\in (0,\delta_0]$,  such that, defining $h_k\coloneqq \rho(\zeta_{j,k},z_{j,k})$ for every $j\in J$ and for every $k\in K$, the mapping
	\[
	\Psi\colon\lambda \mapsto \sum_{j,k} \lambda_{j,k} B_{(\zeta_{k},z_{j,k})}^{\vect{s'}} \Delta_\Omega^{\vect b/q+\vect d/p-\vect s-\vect{s'}}(h_k)
	\]
	is well defined (with locally uniform convergence of the sum) and maps $\ell^{p,q}_0(J,K) $ into $\Ac^{p,q}_{\vect s,0}(D)$ continuously (resp.\ maps $\ell^{p,q}(J,K) $ into $\Ac^{p,q}_{\vect s}(D)$ continuously).
	
	If we may take $(\zeta_{k},z_{j,k})_{j\in J, k\in K}$, for every $\delta_0>0$ as above, in such a way that the corresponding mapping $\Psi$ is onto, then we say that property $(\Lc')^{p,q}_{\vect s,\vect{s'},0}$ (resp.\ $(\Lc')^{p,q}_{\vect s,\vect{s'}}$) holds.
\end{deff}

Arguing as in~\cite[Proposition 4.4 and Corollary 4.7]{Paralipomena}, one may actually show that properties $(\Lc)^{p,q}_{\vect s,\vect{s'},0}$ (resp.\ $(\Lc)^{p,q}_{\vect s,\vect{s'}}$) and $(\Lc')^{p,q}_{\vect s,\vect{s'},0}$ (resp.\ $(\Lc')^{p,q}_{\vect s,\vect{s'}}$) are actually equivalent when $p,q\in [1,\infty]$, and that properties $(\Lc)^{p,q}_{\vect s,\vect{s'},0}$ and $(\Lc)^{p,q}_{\vect s,\vect{s'}}$ are equivalent when $\vect s \succ \vect 0$ (which is a necessary condition for property $(\Lc)^{p,q}_{\vect s,\vect{s'},0}$ to hold). Cf.~also~Theorem~\ref{teo:3} below.

\begin{lem}\label{lem:46}
	Take $p,q\in (0,\infty]$ and $\vect s,\vect{s'}\in \R^r$ such that  property $(\Lc)^{p,q}_{\vect s,\vect{s'},0}$ (resp.\ $(\Lc)^{p,q}_{\vect s,\vect{s'}}$) holds.
	Then, the following hold:
	\begin{itemize}
		\item $\vect s\succ \frac{1}{2 q}\vect m $ (resp.\  $\vect s\Meg \vect 0$ if $q=\infty$) and $\vect s \succ\frac 1 q \vect b+\frac 1 p\vect d +\frac{1}{2 q'} \vect{m'}$;
		
		\item $\vect{s'}\in \frac{1}{\min(p,p')}\vect d-\frac{1}{2 \min(p,p')} \vect{m'}$;
		
		\item $\vect s+\vect{s'}\prec  \frac{1}{\min(1,q)}\vect b+\frac{1}{\min(1,p)}\vect d-\frac{1}{2 q'}\vect m$ or $\vect s+\vect{s'}\meg  \frac{1}{q}\vect b+\frac{1}{\min(1,p)}\vect d$ if $q'=\infty$, and $\vect s+\vect s'\prec \frac 1 q \vect b+\frac 1 p\vect d-\frac{1}{2 q} \vect{m'}$.
	\end{itemize}
\end{lem}

\begin{proof}
	By Proposition~\ref{prop:4}, it will suffice to observe that $ B^{\vect{s'}}_{(0, i e_\Omega)}\in \Ac^{p,q}_{\vect s,0}(D)$ (resp.\ $ B^{\vect{s'}}_{(0, i e_\Omega)}\in \Ac^{p,q}_{\vect s}(D)$), and to show that 
	\[
	B^{\vect{s'}}_{(0, i e_\Omega)}\in \Ac^{p',q'}_{\vect b/\min(1,q)+\vect d/\min(1,p)-\vect s-\vect{s'}}(D).
	\]
	Take  $\delta_0>0$. Then, there is an $F$-$(\delta,4)$-lattice $(\zeta_{k},z_{j,k})_{j\in J, k\in K}$, with $\delta\meg \delta_0$, such that the mapping
	\[
	\ell^{p,q}_0(J,K)\ni \lambda \mapsto \sum_{j,k} \lambda_{j,k} B^{\vect{s'}}_{(\zeta_{k},z_{j,k})} \Delta^{\vect b/q+\vect d/p-\vect s-\vect{s'}}_\Omega(h_k)\in \Ac^{p,q}_{\vect s}(D)
	\]
	is well defined and continuous, where $h_k\coloneqq \rho(\zeta_{k},z_{j,k})$ for every $j\in J$ and for every $k\in K$.
	Observe that the continuity of the mapping $f\mapsto f(0,i e_\Omega)$ on $\Ac^{p,q}_{\vect s}(D)$ implies that there is a constant $C_1>0$ such that
	\[
	\abs*{\sum_{j,k} \lambda_{j,k} B^{\vect{s'}}_{(\zeta_{k},z_{j,k})}(0,i e_\Omega) \Delta^{\vect b/q+\vect d/p-\vect s-\vect{s'}}_\Omega(h_k)}\meg C_1 \norm{\lambda}_{\ell^{p,q}_0(J,K)}
	\]
	for every $\lambda\in \ell^{p,q}_0(J,K)$. 
	Therefore, 
	\[
	\Big(B^{\vect{s'}}_{(0,i e_\Omega)}(\zeta_{j,k},z_{j,k}) \Delta^{\vect b/q+\vect d/p-\vect s-\vect{s'}}_\Omega(h_k)\Big)=\overline{\Big(B^{\vect{s'}}_{(\zeta_{j,k},z_{j,k})}(0,i e_\Omega) \Delta^{\vect b/q+\vect d/p-\vect s-\vect{s'}}_\Omega(h_k)\Big)}\in \ell^{p',q'}(J,K),
	\]
	so that the conclusion follows from  Theorem~\ref{teo:1} and~\cite[Theorem 2.47]{CalziPeloso}.
\end{proof}

\begin{teo}\label{teo:2}
	Take $p,q\in(0,\infty]$ and $\vect s,\vect s'\in \R^r$ such that the following hold:
	\begin{itemize}
		\item $\vect s\succ \frac {1}{2 q}\vect m+\left( \frac{1}{2 \min(1,p)}-\frac{1}{2q} \right)_+\vect m'$;

		\item $\vect s+\vect s'\prec \frac{1}{\min(1,p,q)}\vect b+\frac{1}{\min(1,p)}\vect d-\frac{1}{2 q}\vect m'-\left( \frac{1}{2 \min(1,p)}-\frac{1}{2 q} \right)_+\vect m$.
	\end{itemize}
	Then, properties $(\Lc')^{p,q}_{\vect s,\vect s',0}$ and $(\Lc')^{p,q}_{\vect s,\vect s'}$ hold.
\end{teo}

More precisely, the proof shows that the mapping $\Psi$ of Definition~\ref{38} has a continuous linear section for $\delta$ sufficiently small and $R$ bounded.

\begin{proof}
	Take an $F$-$(\delta,R)$-lattice $(\zeta_{k},z_{j,k})_{j\in J,k\in K}$ on $D$ for some $\delta>0$ and some $R>1$. We shall further assume, as in the proof of Theorem~\ref{teo:1}, that $K=K'\times K''$, and that there is a $(\delta,R)$-lattice $(h_{k''})_{k''\in K''}$ on $\Omega$ such that $h_{k''}= \rho(\zeta_{(k',k'')}, z_{j,(k',k'')})$ for every $j\in J$ and for every $(k',k'')\in K'\times K''$. We shall also write $h_{(k',k'')}$ instead of $h_{k''}$ when it simplifies the notation.
	
	In addition, define 
	\[
	B_{j,k}^{\vect{s''}}\coloneqq B^{\vect{s''}}_{(\zeta_{k},z_{j,k})} 
	\]
	for every $\vect{s''}\in \R^r$ and for every $(j,k)\in J\times K$, in order to simplify the notation. 
	Further, for every $\lambda\in \C^{J\times K}$ define
	\[
	\Psi_+(\lambda)\coloneqq \sum_{j,k} \abs{\lambda_{j,k} B^{\vect {s'}}_{j,k} } \Delta_\Omega^{\vect b/q+\vect d/p-\vect s-\vect{s'}}(h_k)\in [0,\infty]^D,
	\]
	and, for every $\lambda\in \C^{(J\times K)}$,
	\[
	\Psi(\lambda)\coloneqq \sum_{j,k} \lambda_{j,k} B^{\vect {s'}}_{j,k} \Delta_\Omega^{\vect b/q+\vect d/p-\vect s-\vect{s'}}(h_k)\in \Hol(D).
	\]
	We shall first prove that $\lambda \mapsto \norm{\Psi_+(\lambda)}_{\Lc^{p,q}_{\vect s}(D)}$ is a continuous quasi-norm on $\ell^{p,q}(J,K)$. Arguing as in~\cite[Proposition 3.32]{CalziPeloso} and using Proposition~\ref{prop:4}, this will prove that $\Psi$ induces continuous linear mappings $\ell^{p,q}_0(J,K)\to \Ac^{p,q}_{\vect s,0}(D)$ and $\ell^{p,q}(J,K)\to \Ac^{p,q}_{\vect s}(D)$. We shall then prove that these mappings are onto and have continuous linear sections.
	
	\textsc{Step I.} Assume first that $q\meg p\meg 1$. Then, for every $\zeta\in E$ and for every $h\in \Omega$,
	\[
	\norm{\Psi_+(\lambda)_h^{(\zeta)}}_{L^p(F)}^p\meg\sum_{j,k}\abs{\lambda_{j,k}}^p   \Delta_\Omega^{(p/q)\vect b+\vect d-p(\vect s+\vect{s'})}(h_k)  \norm{(B^{\vect {s'}}_{j,k}) _h^{(\zeta)} }_{L^p(F)}^p.
	\]
	In addition,~\cite[Lemma 2.39]{CalziPeloso} shows that there is a constant $C_1>0$ such that 
	\[
	\norm{(B^{\vect {s'}}_{j,k}) _h^{(\zeta)} }_{L^p(F)}^p =C_1 \Delta_\Omega^{p\vect{s'}-\vect d}(h+h_k+\Phi(\zeta-\zeta_{k})) 
	\]
	for every $(j,k)\in J\times K$, for every $\zeta\in E$, and for every $h\in \Omega$.
	Therefore, using the subadditivity of the mapping $x\mapsto x^{q/p}$ on $\R_+$,
	\[
	\begin{split}
		\norm{\Psi_+(\lambda)}_{\Lc^{p,q}_{\vect s}(D)}^q&\meg C_1^{q/p} \sum_{k}  \left( \sum_j \abs{\lambda_{j,k}}^p\right) ^{q/p} \Delta_\Omega^{ \vect b+(q/p)\vect d-q(\vect s+\vect{s'})}(h_k) \times \\
		&\qquad \times\int_\Omega \int_E  \Delta_\Omega^{q\vect{s'}-(q/p)\vect d}(h+h_k+\Phi(\zeta-\zeta_{k}))\,\dd \zeta\, \Delta^{q\vect s}_\Omega(h)\,\dd \nu_\Omega(h).
	\end{split}
	\]
	Now,~\cite[Corollary 2.22 and Lemma 2.32]{CalziPeloso} imply that  there is a constant $C_2>0$ such that
	\[
	\int_\Omega \int_E  \Delta_\Omega^{q\vect{s'}-(q/p)\vect d}(h+h_k+\Phi(\zeta-\zeta_{k}))\,\dd \zeta\, \Delta^{q\vect s}_\Omega(h)\,\dd \nu_\Omega(h)= C_2 \Delta_\Omega^{ q(\vect s+\vect{s'}-\vect b/q-\vect d/p)}(h_k)
	\]
	for every $k\in K$. Hence,
	\[
	\norm{\Psi_+(\lambda)}_{\Lc^{p,q}_{\vect s}(D)}\meg C_1^{1/p} C_2^{1/q} \norm{\lambda }_{\ell^{p,q}(J,K)},
	\]
	whence  our claim in this case.
	
	\textsc{Step II.}  Assume, now, that $q\Meg p\meg 1$. For every $k\in K$, choose $\tau_k\in C_c(E\times \Omega)$ such that
	\[
	\chi_{B_{E\times \Omega}((\zeta_k,h_k),\delta/2)}\meg \tau_k\meg \chi_{B_{E\times \Omega}((\zeta_k,h_k),\delta)}.
	\]
	Observe that by the computations of~\textsc{step I} and~\cite[Corollary 2.49]{CalziPeloso}, there is a constant $C_1'>0$ such that
	\[
	\norm{\Psi_+(\lambda)_h^{(\zeta)}}_{L^p(F)}^p\meg C_1' \int_{E\times \Omega}\sum_{k} \lambda'_k \tau_k(\zeta',h')  \Delta_\Omega^{(p/q)\vect b+\vect d-p(\vect s+\vect{s'})}(h')  \Delta_\Omega^{p\vect{s'}-\vect d}(h+h'+\Phi(\zeta-\zeta'))\,\dd \nu_{E\times \Omega}(\zeta',h') ,
	\]
	where $\lambda'=\left( \sum_{j}\abs{\lambda_{j,k}}^p   \right)_k$.
	Hence, by Minkowski's integral inequality and Young's inequality, and by~\cite[Lemma 2.32]{CalziPeloso},
	\[
	\begin{split}
	&\norm{\Psi_+(\lambda)_h}_{L^{p,q}(F,E)}^p\\
		&\meg C_1' \norm*{\int_{E\times \Omega}\sum_{k} \lambda'_k \tau_k(\zeta',h')  \Delta_\Omega^{(p/q)\vect b+\vect d-p(\vect s+\vect{s'})}(h')  \Delta_\Omega^{p\vect{s'}-\vect d}(h+h'+\Phi(\,\cdot\,-\zeta'))\,\dd \nu_{E\times \Omega}(\zeta',h') }_{L^{q/p}(E)}\\
		&\meg C_1' \int_\Omega  \norm*{\Big( \sum_{k} \lambda'_k \tau_k(\,\cdot\,,h') \Big) * \Delta_\Omega^{p\vect{s'}-\vect d}(h+h'+\Phi(\,\cdot\,)) }_{L^{q/p}(E)}\Delta_\Omega^{[(p/q)+1]\vect b+\vect d-p(\vect s+\vect{s'})}(h')\,\dd \nu_{ \Omega}(h')\\
		&\meg C_1'  \int_{\Omega} \sum_{k''} \norm*{\sum_{k'}\lambda'_{(k',k'')} \tau_{(k',k'')}(\,\cdot\,,h')}_{L^{q/p}(E)}   \norm{\Delta_\Omega^{p\vect{s'}-\vect d}(h+h'+\Phi(\,\cdot\,))}_{L^1(E)}\Delta_\Omega^{[(p/q)+1]\vect b+\vect d-p(\vect s+\vect{s'})}(h') \,\dd \nu_{\Omega}(h') \\
		&\meg C_1''\int_{\Omega} \sum_{k''} \chi_{B_\Omega(h_{k''},\delta)}(h') \norm*{\big(\lambda'_{(k',k'')}\big)_{k'}}_{L^{q/p}(K')}  \Delta_\Omega^{p\vect{s'}-\vect b-\vect d}(h+h')\Delta_\Omega^{\vect b+\vect d-p(\vect s+\vect{s'})}(h') \,\dd \nu_{\Omega}(h') .
	\end{split}
	\]
	If we define $f\coloneqq  \sum_{k''} \chi_{B_\Omega(h_{k''},\delta)}(h') \norm*{\big(\lambda'_{(k',k'')}\big)_{k'}}_{L^{q/p}(K')}$, then it is clear that 
	\[
	\norm{f}_{L^{q/p}(\nu_\Omega)}\meg \nu_\Omega(B_\Omega(e_\Omega,\delta))^{p/q} \norm{ \lambda'}_{L^{q/p}(K)}= \nu_\Omega(B_\Omega(e_\Omega,\delta))^{p/q} \norm{ \lambda}_{L^{p,q}(J,K)}^p.
	\]
	Then,~\cite[Lemma 3.35]{CalziPeloso} implies that there is a constant $C_3>0$ such that
	\[
	\norm{\Psi_+(\lambda)}_{\Lc^{p,q}_{\vect s}(D)}\meg C_3 \norm{\lambda}_{L^{p,q}(J,K)},
	\]
	which completes the proof of our claim in this case.
	
	\textsc{Step III.} Assume, now, that $p,q\Meg 1$. For every $(j,k)\in J\times K$, choose $\tau_{j,k}\in C_c(\Omega)$ so that 
	\[
	\chi_{B((\zeta_{k},z_{j,k}),\delta/2)}\meg \tau_{j,k}\meg \chi_{B((\zeta_{k},z_{j,k}),\delta)},
	\]
	and define 
	\[
	C_4\coloneqq \sup\limits_{(\zeta,h)\in E\times\Omega} \int_F \left( \chi_{B((0,i e_\Omega),\delta)}\right)_h^{(\zeta)}\,\dd \Hc^{m}.
	\]
	Define
	\[
	\Psi'\colon \ell^{p,q}(J,K)\ni \lambda \mapsto \sum_{j,k} \lambda_{j,k} \tau_{j,k} \Delta_\Omega^{\vect b/q+\vect d/p-\vect s}\circ \rho  \in C(D),
	\]
	and let us prove that $\Psi'$ maps continuously %$\ell^{p,q}_0(J,K)$ and 
	$\ell^{p,q}(J,K)$ into %$\Lc^{p,q}_{\vect s,0}(D)$ and 
	$\Lc^{p,q}_{\vect s}(D)$%, respectively
	. Indeed, take $\lambda\in \ell^{p,q}(J,K)$, and observe that
	\[
	\begin{split}
		\norm{(\Psi'( \lambda))_h^{(\zeta)}}_{L^p(F)}&\meg \Delta_\Omega^{\vect b/q+\vect d/p-\vect s}(h) \sum_{k\in K} \norm*{\sum_{j\in J} \abs{\lambda_{j,k}}\left( \chi_{B((\zeta_{k},z_{j,k}),\delta)} \right)_h ^{(\zeta)} }_{L^p(F)}\\
		&\meg C_4^{1/p}\sum_{k\in K}\frac{\Delta_\Omega^{\vect b/q+\vect d/p-\vect s}(h)}{\Delta_\Omega^{\vect d/p}(h_k)}\chi_{B((\zeta_k,h_k),\delta)}(\zeta,h) \norm*{(\lambda_{j,k})_j}_{\ell^p(J)}
	\end{split}
	\]
	for every $h\in \Omega$, so that by~\cite[Corollary 2.49]{CalziPeloso} there is a constant $C_4'>0$ such that
	\[
	\begin{split}
		\norm{\Psi'(\lambda)}_{\Lc^{p,q}_{\vect s}(D)}&\meg  C_4' \norm*{\sum_{k\in K}\chi_{B((\zeta_k,h_k),\delta)} \norm*{(\lambda_{j,k})_j}_{\ell^p(J)}  }_{L^q(\nu_{E\times\Omega})}\\
		&= C_4' \nu_{E\times\Omega}(B_{E\times\Omega}((0,e_\Omega),\delta))^{1/q} \norm{\lambda}_{\ell^{p,q}(J,K)}.
	\end{split}
	\]
	Thus, $\Psi'$ induces a continuous linear mapping $\ell^{p,q}(J,K)\to \Lc^{p,q}_{\vect s}(D)$.
	%In addition, since   $\Psi'(\C^{(J\times K)})\subseteq C_c(D)$, we also see that $\Psi'$ induces a continuous linear mapping $\ell^{p,q}_0(J,K)\to \Lc^{p,q}_{\vect s,0}(D)$.
	
	Now, observe that~\cite[Theorem 2.47 and Corollary 2.49]{CalziPeloso} imply that there is a constant $C_4''>0$ such that
	\[
	\begin{split}
		\abs{\Psi_+(\lambda)_h(\zeta,x)}&\meg  C_4'' \int_\Omega \int_\Nc \Psi'(\abs{\lambda})_{h'}(\zeta',x') \abs*{\left(B_{(\zeta',x'+i\Phi(\zeta')+i h')}^{\vect{s'}} \right)_{h}(\zeta,x) } \,\dd (\zeta',x')\Delta_\Omega^{\vect b+\vect d-\vect{s'}}(h')\,\dd \nu_\Omega(h')  \\
		&= C_4'' \int_\Omega \left( \Psi'(\abs{\lambda})_{h'}* \abs*{\left(B_{(0, i h')}^{\vect{s'}} \right)_{h}} \right) (\zeta,x) \Delta_\Omega^{\vect b+\vect d-\vect{s'}}(h')\,\dd \nu_\Omega(h')  \\
	\end{split}
	\]
	for every $\lambda\in \ell^{p,q}(J,K)$, for every $h\in \Omega$, and for every $(\zeta,x)\in \Nc$.
	Therefore, Minkowski's integral inequality, Young's inequality (applied twice),  and~\cite[Lemma 2.39]{CalziPeloso} show that there is a constant $C_5>0$ such that
	\[
	\norm{\Psi_+(\lambda)_h}_{L^{p,q}(F,E)}\meg C_5 \int_\Omega \norm{\Psi'(\abs{\lambda})_{h'}}_{L^{p,q}(F,E)}\Delta_\Omega^{\vect{s'}-(\vect b+\vect d)}(h+h')\Delta_\Omega^{\vect b+\vect d-\vect{s'}}(h') \,\dd \nu_\Omega(h')
	\]
	for every $\lambda\in \ell^{p,q}(J,K)$, and for every $h\in \Omega$.
	
	Define 
	\[
	T'\colon  f \mapsto \Delta_\Omega^{\vect s} \int_\Omega f(h')\Delta_\Omega^{\vect{s'}-(\vect b+\vect d)}(\,\cdot\,+h')\Delta_\Omega^{\vect b+\vect d-\vect s-\vect{s'}}(h') \,\dd \nu_\Omega(h').
	\]
	so that $T'$ induces an endomorphism of %$L^q_0(\nu_\Omega)$ and 
	$L^q(\nu_\Omega)$ by~\cite[Lemma 3.35]{CalziPeloso}. Then,
	\[
	\begin{split}
		\norm{\Psi_+(\lambda)}_{\Lc^{p,q}_{\vect s}(D)}&\meg C_5 \norm{T'}_{\Lin(L^q(\nu_\Omega))} \norm{\Psi'(\abs{\lambda})}_{\Lc^{p,q}_{\vect s}(D)}\\
		&\meg C_5 \norm{T'}_{\Lin(L^q(\nu_\Omega))} C_4' \nu_{E\times\Omega}(B_{E\times\Omega}((0,e_\Omega),\delta))^{1/q} \norm{\lambda}_{\ell^{p,q}(J,K)}.
	\end{split}
	\] 
	%Since clearly $\Psi(\C^{(J\times K)})\subseteq \Lc^{p,q}_{\vect s,0}(D)$, 
	Our claim then follows also in this case.

	\textsc{Step IV.} Finally, assume that $p\Meg 1\Meg q$. For every $(j,k)\in J\times K$, choose $\tau'_{j,k}\in C_c(F)$ so that
	\[
	(\chi_{B((\zeta_k,z_{j,k}),\delta/2)})^{(\zeta_k)}_{h_k}\meg \tau'_{j,k}\meg (\chi_{B((\zeta_k,z_{j,k}),\delta)})^{(\zeta_k)}_{h_k},
	\]
	and define 
	\[
	\Psi''\colon \ell^{p,q}(J,K)\ni \lambda \mapsto  \left(\sum_{j} \lambda_{j,k} \tau'_{j,k} \Delta_\Omega^{\vect d/p}(h_k)\right)_k\in C_c(F)^K.
	\]
	As in~{step III}, one may show that $\Psi''$ induces a continuous linear mapping of $\ell^{p,q}(J,K)$ into $\ell^q(K; L^p(F))$. In addition, by means of~\cite[Theorem 2.47]{CalziPeloso} we see that there is a constant $C_6>0$ such that
	\[
	\begin{split}
		\abs{\Psi_+(\lambda)_h^{(\zeta)}(x)}&\meg  C_6 \sum_{k\in K}\Delta_\Omega^{\vect b/q+\vect d-\vect s-\vect{s'}}(h_k)\int_F \Psi''(\abs{\lambda})_k (x') \abs*{\left(B_{(\zeta_k,x'+i\Phi(\zeta_k)+i h_k)}^{\vect{s'}} \right)_{h}^{(\zeta)}(x) } \,\dd x'\\
		&= C_6 \sum_{k\in K}\Delta_\Omega^{\vect b/q+\vect d-\vect s-\vect{s'}}(h_k)\left(  \Psi''(\abs{\lambda})_k *\abs*{\left(B_{(\zeta_k,i\Phi(\zeta_k)+i h_k)}^{\vect{s'}} \right)_{h}^{(\zeta)} }\right)  (x)\\
	\end{split}
	\]
	for every $\lambda\in \ell^{p,q}(J,K)$, for every $(\zeta,h)\in E\times\Omega$, and for every $x\in F$. Then, by  Minkowski's inequality, Young's inequality, and~\cite[Lemma 2.39]{CalziPeloso}, there is a constant $C_6'>0$ such that
	\[
	\norm{\Psi_+(\lambda)_h^{(\zeta)}}_{L^{p}(F)}\meg C_6' \sum_{k\in K} \norm{\Psi''(\abs{\lambda})_k}_{L^{p}(F)}\Delta_\Omega^{\vect b/q+\vect d-\vect s-\vect{s'}}(h_k) \Delta_\Omega^{\vect{s'}-\vect d}(h+h_k+\Phi(\zeta-\zeta_k))
	\]
	for every $\lambda\in \ell^{p,q}(J,K)$ and for every $(\zeta,h)\in E\times\Omega$.
	Then, by the subadditivity of the mapping $x\mapsto x^q$ on $\R_+$, 
	\[
	\begin{split}
		\norm{\Psi_+(\lambda)}_{\Lc^{p,q}_{\vect s}(D)}^q&\meg C_6'^q \sum_{k} \norm{\Psi''(\abs{\lambda})_k}_{L^{p}(F)}^q \Delta_\Omega^{\vect b+q(\vect d-\vect s -\vect{s'})}(h_{k}) \\
			&\qquad \times\int_{E\times\Omega} \Delta_\Omega^{q\vect s-\vect b}(h) \Delta_\Omega^{q(\vect{s'}-\vect d)}(h+h_{k}+\Phi(\zeta-\zeta_k)) \,\dd\nu_{E\times\Omega}(\zeta,h).
	\end{split}
	\]
	Now, by homogeneity,
	\[
	\begin{split}
	&\Delta_\Omega^{\vect b+q(\vect d-\vect s-\vect{s'})}(h_{k})\int_{E\times\Omega} \Delta_\Omega^{q\vect s-\vect b}(h) \Delta_\Omega^{q(\vect{s'}-\vect d)}(h+h_{k}+\Phi(\zeta-\zeta_k)) \,\dd\nu_{E\times\Omega}(\zeta,h)\\
		&\qquad=\int_{E\times\Omega} \Delta_\Omega^{q\vect s-\vect b}(h) \Delta_\Omega^{q(\vect{s'}-\vect d)}(h+e_\Omega+\Phi(\zeta)) \,\dd\nu_{E\times\Omega}(\zeta,h)
	\end{split}
	\]
	for every $k\in K$, and the last integral is finite by~\cite[Corollary 2.22 and Lemma 2.32]{CalziPeloso}.
	Therefore, there is a constant $C_7>0$ such that
	\[
	\norm{\Psi_+(\lambda)}_{\Lc^{p,q}_{\vect s}(D)}\meg C_7 \norm{\Psi''(\abs{\lambda})}_{\ell^q(K;L^p(F))}, 
	\]
	whence our claim also in this case. %, since $\Psi(\C^{(J\times K)})\subseteq \Lc^{p,q}_{\vect s,0}(D)$.

%	Hence,~\cite[Corollary 2.22]{CalziPeloso}  implies that there is a constant $C_7>0$ such that 
%	\[
%	\norm{\Psi(\lambda)}_{\Lc^{p,q}_{\vect s}(D)}\meg C_7 \norm*{\lambda }_{\ell^{p,q}(J,K)}.
%	\]
%	In order to conclude the proof of our claim in this case, it suffices to observe that $\Psi(\C^{(J\times K)})\subseteq L^{p,q}_{\vect s,0}(D)$.
	
	\textsc{Step V.} 	Put a well-ordering on $J\times K$, and define 
	\[
	U_{j,k} \coloneqq B((\zeta_{k},z_{j,k}),R\delta)\setminus \left(\bigcup_{(j',k')<(j,k)} B((\zeta_{k'},z_{j',k'}),R\delta) \right)
	\]
	for every $(j,k)\in J\times K$, so that $(U_{j,k})_{(j,k)\in J\times K}$ is a Borel measurable partition of $D$ (since $J$ and $K$ are countable). 
	In addition, define $c_{j,k}\coloneqq c \nu_D(U_{j,k})$ for every $(j,k)\in J\times K$, where $c>0$ is defined so that $P_{\vect s'} f=c\int_D f(\zeta,z) B_{(\zeta,z)}^{\vect s}\Delta^{-\vect s'}(\rho(\zeta,z))\,\dd \nu_D(\zeta,z)$ for every $f\in C_c(D)$. Then, 
	\[
	c\,\nu_D(B((0,i e_\Omega),\delta))\meg c_{j,k}\meg c\,\nu_D(B((0,i e_\Omega),R\delta))
	\]
	for every $(j,k)\in J\times K$.
	Then, define 
	\[
	S\colon \Ac^{p,q}_{\vect s}(D)\ni f \mapsto \left(c_{j,k}\Delta_\Omega^{\vect s-\vect b/q+\vect d/p}(h_k) f(\zeta_{k},z_{j,k}) \right)\in \ell^{p,q}(J,K),
	\]
	so that Theorem~\ref{teo:1} shows that $S$ is well defined and continuous, and maps $\Ac^{p,q}_{\vect s,0}(D)$ into $\ell^{p,q}_0(J,K)$. Define $S'\coloneqq \Psi S$.
	Then, Proposition~\ref{prop:3} implies that, for every $ f\in \Ac^{p,q}_{\vect s}(D) $,
	\[
	\begin{split}
		f- S' f&= c\sum_{j,k}\int_{ U_{j,k} } \Big(f(\zeta',z') B^{\vect{s'}}_{(\zeta',z')} \Delta_\Omega^{-\vect{s'}}(\rho(\zeta',z'))  -f(\zeta_{k},z_{j,k}) B^{\vect{s'}}_{(\zeta_{k},z_{j,k})} \Delta_\Omega^{-\vect{s'}}(h_k) \Big) \,\dd \nu_D(\zeta',z').
	\end{split}
	\]
	Hence,~\cite[Theorem 2.47, Corollary 2.49, and Lemma 3.25]{CalziPeloso} imply that there are $R'_0>0$ and $C_8>0$ such that
	\[
	\begin{split}
		&\abs{(f-S' f)(\zeta,z)}\meg C_8 R\delta\sum_{j,k} c_{j,k}\sup\limits_{(\zeta',z')\in B((\zeta_{k},z_{j,k}), R\delta+R'_0)} \abs{f(\zeta',z')} \abs{B^{\vect{s'}}_{(\zeta_{k},z_{j,k})}(\zeta,z)} \Delta_\Omega^{-\vect{s'}}(h_k) 
	\end{split}
	\]
	for every $(\zeta,z)\in D$. 
	Fix an $F$-$(1,4)$-lattice $(\zeta'_{k'},z'_{j',k'})_{j'\in J', k'\in K'}$ on $D$, and observe that the proof of~\cite[Proposition 2.56]{CalziPeloso}, together with~\cite[Theorem 2.47 and Corollary 2.49]{CalziPeloso} again, implies that there is a constant $C_9>0$ such that
	\[
	\begin{split}
		\abs{(f-S' f)(\zeta,z)}\meg C_9 R \delta\sum_{j',k'}\sup\limits_{(\zeta',z')\in B\big((\zeta'_{k'},z'_{j',k'}), R\delta+R'_0+4\big)} \abs{f(\zeta',z')} \abs{B^{\vect{s'}}_{(\zeta'_{k'},z'_{j',k'})}(\zeta,z)} \Delta_\Omega^{-\vect{s'}}(h'_{k'}) 
	\end{split}
	\]
	for every $(\zeta,z)\in D$, where $h'_{k'}= \rho(\zeta'_{k'},z'_{j',k'})$ for  every $j'\in J'$ and for every $k'\in K'$.
	Hence, Theorem~\ref{teo:1} and  the preceding steps show that there is a constant $C_{10}>0$ such that, if $R\meg R_0$ and $\delta\meg 1$, then
	\[
	\norm{f-S' f}_{\Ac^{p,q}_{\vect s}(D)}\meg C_{10} \delta \norm{f}_{\Ac^{p,q}_{\vect s}(D)}.
	\]
	Take $\delta_0>0$ so that $C_{10} \delta_0\meg \frac 1 2$, and assume that $\delta\meg \delta_0$. Then, 
	\[
	\norm*{\sum_{j\Meg k}   (I-S')^j f }_{\Ac^{p,q}_{\vect s}(D)}^{\min(1,p,q)}\meg \sum_{j\Meg k} 2^{-\min(1,p,q) j} \norm{f}_{\Ac^{p,q}_{\vect s}(D)}^{\min(1,p,q)}
	\]
	for every $k\in \N$, so that $\sum_{j\in \N} (I- S')^j$ induces well defined endomorphisms of $\Ac^{p,q}_{\vect s,0}(D) $ and $\Ac^{p,q}_{\vect s}(D)$,  which are inverses of $S'$. 
	Hence, 
	\[
	\Psi'''\coloneqq S\sum_{j\in \N} (I-S')^j
	\]
	induces well defined and continuous linear mappings from $\Ac^{p,q}_{\vect s,0}(D)$ into $\ell^{p,q}_0(J,K)$ and from $\Ac^{p,q}_{\vect s}(D)$ into $\ell^{p,q}(J,K)$, and $\Psi \Psi'''= S'\sum_{j\in \N}(I-S')^j=I$. The proof is complete.
\end{proof}

Arguing as in~\cite[Proposition 3.37]{CalziPeloso}, one may prove the following result.

\begin{prop}\label{prop:9}
	Take $p,q\in (0,\infty]$ and $\vect s,\vect s'\in \R^r$ such that property $(\Lc')^{p,q}_{\vect s,\vect s',0}$ holds. Then, the sesquilinear mapping
	\[
	(f,g)\mapsto \int_D f \overline g (\Delta^{-\vect s'}\circ \rho)\,\dd \nu_D
	\]
	induces an antilinear isomophism of $\Ac^{p',q'}_{\vect b/\min(1,q)+\vect d/\min(1,p)-\vect s-\vect s'}(D)$ onto $\Ac^{p,q}_{\vect s,0}(D)'$.
\end{prop}

\subsection{Boundary Values}

\begin{lem}\label{lem:2}
	 The continuous linear mappings $\Sc_{\overline{\Omega'}}(\Nc)\ni u \mapsto u(\zeta,\,\cdot\,)\in\Sc_{\overline{\Omega'}}(F)$, $\zeta \in E$, induce uniquely determined continuous linear mappings $\Sc'_{\overline{\Omega'}}(\Nc)\ni u \mapsto u(\zeta,\,\cdot\,)\in \Sc'_{\overline{\Omega'}}(F)$ such that the following hold:
	 \begin{itemize}
	 	\item[\textnormal{(1)}] for every $u\in \Sc'_{\overline{\Omega'}}(\Nc)$, for every $\zeta\in E$, and for every $\phi \in \Sc(F')$ supported in $\overline{\Omega'}$,
	 	\[
	 	(u*\Fc_\Nc^{-1}(\phi))(\zeta,\,\cdot\,)= u(\zeta,\,\cdot\,)*\Fc_F^{-1}(\phi);
	 	\]
	 	
	 	\item[\textnormal{(2)}]  for every $u\in \Sc'_{\overline{\Omega'}}(\Nc)$, the mapping $E \ni \zeta \mapsto u(\zeta,\,\cdot\,)\in \Sc'_{\overline{\Omega'}}(F)$ is of class $C^\infty$.
	 \end{itemize}
\end{lem}

As a consequence, we shall identify each $u\in \Sc'_{\overline{\Omega'}}(\Nc)$ with a map $E\ni\zeta \mapsto u(\zeta,\,\cdot\,)\in \Sc'_{\overline{\Omega'}}(F)$ of class $C^\infty$.

\begin{proof}
	Fix $\phi \in \Sc(F')$ supported in $\overline{\Omega'}$, and define $\psi\coloneqq \Fc_\Nc^{-1}(\phi)$ and $\psi'\coloneqq \Fc_F^{-1}(\phi)$. Then,  $\pi_\lambda(\psi)=\phi(\lambda) P_{\lambda,0}$ and $\pi_\lambda(\delta_0\otimes \psi')=\phi(\lambda) I$ for every $\lambda\in \Lambda_+$, so that
	\[
	u*\psi=u*(\delta_0\otimes\psi')
	\] 
	for every $u\in \Sc_{\overline{\Omega'}}(\Nc)$, that is,
	\[
	(u*\psi)(\zeta,\,\cdot\,)= u(\zeta,\,\cdot\,)*\psi'
	\]
	for every $\zeta\in E$. In particular,
	\[
	\langle u(\zeta,\,\cdot\,)\vert \psi'\rangle=  (u(\zeta,\,\cdot\,)*\psi'^*)(0)=(u*\psi^*)(\zeta,0)=\langle u\vert \psi((\zeta,0)^{-1}\,\cdot\,)\rangle,
	\]
	so that, by the arbitrariness of $\phi$, the mapping $\Sc_{\overline{\Omega'}}(\Nc)\ni u \mapsto u(\zeta,\,\cdot\,)\ni \Sc_{\overline{\Omega'}}(F)$ is continuous for the topologies induced by $\Sc_{\overline{\Omega'}}'(\Nc)$ and $\Sc_{\overline{\Omega'}}'(F)$ on $\Sc_{\overline{\Omega'}}(\Nc)$ and $ \Sc_{\overline{\Omega'}}(F)$, respectively. Since $\Sc_{\overline{\Omega'}}(\Nc)$ is dense in $\Sc_{\overline{\Omega'}}'(\Nc)$ (because the conjugate of $\Sc_{\overline{\Omega'}}(\Nc)$ is reflexive and the polar of $\Sc_{\overline{\Omega'}}(\Nc)$ in the conjugate of $\Sc_{\overline{\Omega'}}(\Nc)$ is $\Set{0}$), and since $\Sc_{\overline{\Omega'}}'(F)$ is complete, the first assertion follows, as well as (1).
	Assertion (2) is a consequence of the fact that the mapping $\zeta \mapsto \langle u(\zeta,\,\cdot\,)\vert \psi'\rangle=(u*\psi^*)(\zeta,0)$ is of class $C^\infty$ on $E$ for every $\phi$ (cf.~\cite[p.\ 59, Lemma II]{Schwartz}).	
\end{proof}

\begin{lem}\label{lem:3}
	Take $p,q,p_2,q_2\in (0,\infty]$ with $p\meg p_2$ and $q\meg q_2$, and a bounded subset $B$ of $\Sc_{\overline{\Omega'}}(\Nc)$ such that $\Fc_\Nc B$ is bounded in $C^\infty_c(\Omega')$. For every $\psi \in B$ and for every $t\in T_+$, define $\psi_t\coloneqq \Fc_\Nc^{-1}((\Fc_\Nc \psi)(\,\cdot\, t^{-1}))$. Then, there is a constant $C>0$ such that
	\[
	\norm{u*\psi_t}_{L^{p_2,q_2}(F,E)}\meg C \Delta^{(1/q_2-1/q)\vect b+(1/p_2-1/p)\vect d}(t)  \norm{u*\psi_t}_{L^{p,q}(F,E)}
	\]
	and
	\[
	\norm{u*\psi_t*\psi'_{t'}}_{L^{p,q}(F,E)}\meg C\norm{u*\psi_t}_{L^{p,q}(F,E)}
	\]
	for every $u\in \Sc'(\Nc)$, for every $\phi,\phi'\in B$, and for every $t,t'\in T_+$.
\end{lem}

\begin{proof}
	\textsc{Step I.} It will suffice to prove the first assertion when $p_2=q_2=\infty$ and $t$ is the identity of $T_+$, by H\"older's interpolation and homogeneity. Arguing by approximation as in the proof of~\cite[Corollary 4.7]{CalziPeloso}, we may further assume that $u\in \Sc(\Nc)$. 
	Then, set $\ell\coloneqq \min(1,p,q)$ and take $\tau \in \Sc_{\overline{\Omega'}}(\Nc)$ so that $\psi*\tau=\psi$ for every $\psi\in B$, and observe that
	\[
	\begin{split}
	\abs{(u*\psi)(\zeta,x)}&\meg \norm{u*\psi}^{1-\ell}_{L^\infty(\Nc)}\int_\Nc \abs{(u*\psi)(\zeta',x')}^{\ell} \abs{\tau((\zeta',x')^{-1}(\zeta,x))}\,\dd (\zeta',x')\\
		&\meg \norm{u*\psi}^{1-\ell}_{L^\infty(\Nc)}\norm{u*\psi}_{L^{p,q}(F,E)} ^\ell \norm{\tau}_{L^{(p/\ell)',(q/\ell)'}(F,E)}
	\end{split}
	\]
	for every $(\zeta,x)\in \Nc$, whence the first assertion.
	
	\textsc{Step II.} The second assertion follows from~\cite[Corollary 4.10]{CalziPeloso} and Lemma~\ref{lem:2}.
\end{proof}

\begin{deff}\label{def:1}
	Take $p,q\in (0,\infty]$ and $\vect s\in \R^r$. Take a $(\delta,R)$-lattice $(\lambda_k)_{k\in K}$ in $\Omega'$ for some $\delta>0$ and some $R>1$, and fix a bounded family $(\phi_k)_{k\in K}$ of positive elements of $C^\infty_c(\Omega')$ such that $\sum_k \phi_k(\,\cdot\, t_k^{-1})\Meg 1$ on $\Omega'$, where $t_k\in T_+$ and $\lambda_k=e_{\Omega'}\cdot t_k$ for every $k\in K$. Define $\psi_k\coloneqq \Fc_\Nc^{-1}(\phi_k(\,\cdot\, t_k^{-1}))$. Then, we define $\mathring \Bc_{p,q}^{\vect s}(\Nc,\Omega)$ (resp.\ $\Bc_{p,q}^{\vect s}(\Nc,\Omega)$) as the space of the $u\in \Sc'_{\overline{\Omega'}}(\Nc)$ such that
	\[
	(\Delta^{\vect s}_{\Omega'}(\lambda_k) \,u*\psi_k)\in \ell^q_0(K;L^{p,q}_0(F,E)) \qquad \text{(resp.\ $(\Delta^{\vect s}_{\Omega'}(\lambda_k)\, u*\psi_k)\in \ell^q(K;L^{p,q}(F,E))$)},
	\]
	endowed with the corresponding topology.\footnote{One may prove directly that this definition does not depend on the choice of $(\lambda_k)$ and $(\phi_k)$, arguing as in the proof of~\cite[Lemma 4.14]{CalziPeloso} and using Lemma~\ref{lem:3}. Nonetheless, this follows from Remark~\ref{oss:3}, at least for $\Bc_{p,q}^{\vect s}(\Nc,\Omega)$.}
\end{deff}

In particular, $\Bc^{\vect s}_{p,p}(\Nc,\Omega)=B^{\vect s}_{p,p}(\Nc,\Omega)$ and $\mathring \Bc^{\vect s}_{p,p}(\Nc,\Omega)=\mathring B^{\vect s}_{p,p}(\Nc,\Omega)$ for every $p\in (0,\infty]$ and for every $\vect s\in \R$.
We now propose a different interpretation of $\Bc_{p,q}^{\vect s}(\Nc,\Omega)$ which is particularly useful in certain situations. 

\begin{oss}\label{oss:3}
	By Lemma~\ref{lem:2}, $\Bc_{p,q}^{\vect s}(\Nc,\Omega)$ may be equivalently defined as  the  space of $u\in \Sc'_{\overline{\Omega'}}(\Nc)$ such that 
	\[
	\norm*{\zeta \mapsto \norm{u(\zeta,\,\cdot\,)}_{B_{p,q}^{\vect s}(F,\Omega)}}_{L^q(E)}<\infty,
	\]
	where $\norm{\,\cdot\,}_{B_{p,q}^{\vect s}(F,\Omega)}$ denotes a fixed quasi-norm on $B_{p,q}^{\vect s}(F,\Omega)$. 
\end{oss}

\begin{prop}\label{prop:10}
	Take $p_1,p_2,q_1,q_2\in (0,\infty]$ and $\vect s_1,\vect s_2\in \R^r$ so that
	\[
	p_1\meg p_2, \qquad q_1\meg q_2 \qquad \text{and} \qquad \vect s_2=\vect s_1+\left(\frac{1}{q_1}-\frac{1}{q_2}\right)\vect b+\left(\frac{1}{p_1}-\frac{1}{p_2}\right)\vect d.
	\]
	Then, there are continuous inclusions $\mathring \Bc^{\vect s_1}_{p_1,q_1}(\Nc,\Omega)\subseteq \mathring\Bc^{\vect s_2}_{p_2,q_2}(\Nc,\Omega)$ and $\Bc^{\vect s_1}_{p_1,q_1}(\Nc,\Omega)\subseteq \Bc^{\vect s_2}_{p_2,q_2}(\Nc,\Omega)$.
	
	In addition, the mappings $\Bc^{\vect s}_{p,q}(\Nc,\Omega)\ni u\mapsto u(\zeta,\cdot\,)\in B^{\vect s+\vect b/q}_{p,q}(F,\Omega)$ are equicontinuous for every $p,q\in (0,\infty]$ and for every $\vect s \in \R^r$.
\end{prop}

\begin{proof}
	This is a consequence of Lemmas~\ref{lem:2} and~\ref{lem:3}, and of the continuous inclusion $\ell^{q_1}(K)\subseteq \ell^{q_2}(K)$, which holds for every set $K$.
\end{proof}

\begin{prop}\label{prop:12}
	Take $p,q\in (0,\infty]$ and $\vect s\in \R^r$. Then, $\mathring \Bc^{\vect s}_{p,q}(\Nc,\Omega)$ and $\Bc^{\vect s}_{p,q}(\Nc,\Omega)$ are complete and embed continuously into $\Sc'_{\overline{\Omega'}}(\Nc)$.
\end{prop}

\begin{proof}
	Since $\Bc^{\vect s}_{p,q}(\Nc,\Omega)\subseteq \Bc_{\infty,\infty}^{\vect s+\vect b/q+\vect d/p}(\Nc,\Omega)=B_{\infty,\infty}^{\vect s+\vect b/q+\vect d/p}(\Nc,\Omega)$ continuously by Proposition~\ref{prop:10}, by~\cite[Proposition 7.12]{Besov} we see that $\Bc^{\vect s}_{p,q}(\Nc,\Omega)$ embeds continuously into $ \Sc'_{\overline{\Omega'}}(\Nc)$. Completeness is then proved as in~\cite[Proposition 4.16]{CalziPeloso}.
\end{proof}

\begin{deff}
	Take $(\psi_k)$ as in Definition~\ref{def:1}, and assume further that $\sum_k (\Fc_\Nc \psi_k)^2=1$ on $\Omega'$. Then, the mapping
	\[
	(u,u')\mapsto \sum_k \langle u*\psi_k\vert u*\psi_k\rangle=\int_E \langle u(\zeta,\,\cdot\,)\vert u'(\zeta,\,\cdot\,)\rangle\,\dd \zeta
	\]
	induces a well defined continuous sesquilinear form on $\Bc^{\vect s}_{p,q}(\Nc,\Omega) \times \Bc^{-\vect s}_{p',q'}(\Nc,\Omega)$, for every $p,q\in[1,\infty]$ and for every $\vect s\in \R^r$, which does not depend on the choice of $(\psi_k)$ (cf.~\cite[the proof of Proposition 4.20]{CalziPeloso}).
\end{deff}

In particular, by Proposition~\ref{prop:10}, there is a canonical sesquilinear form on   
\[
\Bc^{\vect s}_{p,q}(\Nc,\Omega) \times \Bc^{-\vect s-(1/q-1)_+\vect b-(1/p-1)_+\vect d}_{p',q'}(\Nc,\Omega)
\]
for every $p,q\in (0,\infty]$ and for every $\vect s\in \R^r$. We denote by $\sigma^{\vect s}_{p,q}$ the weak topology 
\[
\sigma(\Bc^{\vect s}_{p,q}(\Nc,\Omega),\mathring\Bc^{-\vect s-(1/q-1)_+\vect b-(1/p-1)_+\vect d}_{p',q'}(\Nc,\Omega)).
\]

\begin{lem}\label{lem:4}
	Take $p,q\in [1,\infty]$ and $\vect s\succ \frac{1}{q'}\vect b+\frac{1}{p'}\vect d+\frac{1}{2 q}\vect m'$. For every $(\zeta,z)\in D$, define $S_{(\zeta,z)}\coloneqq c \left(B^{\vect b+\vect d}_{(\zeta,z)}\right)_0$, where $c\neq 0$ is chosen so that $f(\zeta,z)=\langle f_0\vert S_{(\zeta,z)}\rangle$ for every $f\in A^{2,\infty}_{\vect 0}(D)$. Then,  the following hold:
	\begin{enumerate}
		\item[\textnormal{(1)}] the $\Delta^{\vect s-\vect b/q'-\vect d/p'}_\Omega(\rho(\zeta,z)) S_{(\zeta,z)}$, as $(\zeta,z)$ runs through $D$, stay in a bounded subset of $\mathring \Bc^{\vect s}_{p,q}(\Nc,\Omega)$;
		
		\item[\textnormal{(2)}] for every $u\in  \Bc^{-\vect s}_{p',q'}(\Nc,\Omega)$, the mapping $(\zeta,z)\mapsto \langle u \vert S_{(\zeta,z)}\rangle$ is holomorphic on $D$.
	\end{enumerate}
\end{lem}

The proof is analogous to that of~\cite[Lemma 5.1]{CalziPeloso} and is omitted.

\begin{prop}\label{prop:11}
	Take $p,q\in(0,\infty]$ and $\vect s \succ \frac 1 q \vect b+\frac 1 p \vect d+\frac{1}{2 q'}\vect m'$. Define a continuous linear mapping (cf.~Lemma~\ref{lem:4})
	\[
	\Ec \colon \Bc^{-\vect s}_{p,q}(\Nc,\Omega)\ni u \mapsto [(\zeta,z)\mapsto \langle u\vert S_{(\zeta,z)}\rangle]\in A^{\infty,\infty}_{\vect s-\vect b/q-\vect d/p}(D),
	\]
	Then, the following hold:
	\begin{enumerate}
		\item[\textnormal{(1)}]  for every $u\in \Bc^{\vect s}_{p,q}(\Nc,\Omega)$ and for every $(\zeta,z)\in D$ such that $u(\zeta,\,\cdot\,)\in B^{\vect s_2}_{p_2,q_2}(F,\Omega)$ for some $p_2,q_2\in (0,\infty]$ and for some $\vect s_2\succ \frac{1}{p_2} \vect d+\frac{1}{2 q_2'}\vect m'$, %\footnote{Notice that, by Proposition~\ref{prop:10}, $u(\zeta,\,\cdot\,)\in B^{\vect s+\vect b/q}_{p,q}(F,\Omega)$, so that this is the case if }
		\[
		(\Ec u)(\zeta,z)=[\Ec (u(\zeta,\,\cdot\,))](z-i\Phi(\zeta));
		\]

		\item[\textnormal{(2)}] the linear mappings $u\mapsto (\Ec u)_h$, as $h$ runs through $\Omega$, induce equicontinuous endomorphisms of  $\Bc^{-\vect s}_{p,q}(\Nc,\Omega)$ and $\mathring \Bc^{-\vect s}_{p,q}(\Nc,\Omega)$;
		
		\item[\textnormal{(3)}] the mapping
		\[
		\Omega\cup \Set{0}\ni h \mapsto (\Ec u)_h\in \Bc^{-\vect s}_{p,q}(\Nc,\Omega)
		\]
		is continuous if $u\in \mathring \Bc^{-\vect s}_{p,q}(\Nc,\Omega)$, and is continuous in the weak topology $\sigma^{-\vect s}_{p,q}$ if $u\in \Bc^{-\vect s}_{p,q}(\Nc,\Omega)$.
	\end{enumerate}
\end{prop}

\begin{proof}
	(1) It will suffice to prove the assertion when $u$ is replaced by $u*\psi$ for some $\psi\in \Sc_{\overline{\Omega'}}(\Nc)$ with $\Fc_\Nc \psi\in C^\infty_c(\Omega')$. Then,
	\[
	[\Ec (u*\psi)]_h=(u*\psi)*(S_{(0,i h)}*\psi'),
	\]
	where $\psi'\in  \Sc_{\overline{\Omega'}}(\Nc)$ and $\psi=\psi*\psi'$.  It then suffices to apply Lemma~\ref{lem:2}.
	
	(2)--(3). The proof is analogous to that of~\cite[Theorem 5.2]{CalziPeloso} (replacing~\cite[Corollary 4.10]{CalziPeloso} with Lemma~\ref{lem:3}).
\end{proof}

\begin{prop}\label{prop:6}
	Take $p,q\in (0,\infty]$ and $\vect s\in \R^r$. Then, the following hold:
	\begin{enumerate}
		\item[\textnormal{(1)}] $\Sc_{\overline{\Omega'}}(\Nc) $ embeds continuously as a dense subspace of $\mathring\Bc^{\vect s}_{p,q}(\Nc,\Omega)$; 
		
		\item[\textnormal{(2)}] for every $\vect s'\in \C^r$, the mapping  $\Sc_{\overline{\Omega'}}(\Nc)\ni\phi\mapsto \phi*I^{\vect s'}_\Omega\in \Sc_{\overline{\Omega'}}(\Nc)$ induces isomorphisms of  $\mathring\Bc^{\vect s}_{p,q}(\Nc,\Omega)$ and $ \Bc^{\vect s}_{p,q}(\Nc,\Omega)$ onto $\mathring\Bc^{\vect s+\Re\vect s'}_{p,q}(\Nc,\Omega)$ and $\Bc^{\vect s+\Re\vect s'}_{p,q}(\Nc,\Omega)$, respectively;
		
		\item[\textnormal{(3)}] if $\vect s \succ \frac 1 q\vect b+\frac 1 p \vect d+\frac{1}{2 q'}\vect m'$ and $\vect s \succ \frac{1}{2 q}\vect m$ (resp.\ $\vect s \Meg \vect 0$ if $q=\infty$), then there are continuous inclusions
		\[
		\Ec(\Sc_{\overline {\Omega'}}(\Nc))\subseteq \Ac^{p,q}_{\vect s,0}(D)\subseteq \widetilde \Ac^{p,q}_{\vect s,0}(D)  \qquad \text{(resp.\ $\Ec(\Sc_{\overline {\Omega'}}(\Nc))\subseteq \Ac^{p,q}_{\vect s}(D)\subseteq \widetilde \Ac^{p,q}_{\vect s}(D)$).}
		\]
		
		\item[\textnormal{(4)}] if $\vect s \succ \frac 1 p \vect d+\frac{1}{2 q'}\vect m'$ and $A^{p,q}_{\vect s,0}(T_\Omega)=\widetilde  A^{p,q}_{\vect s,0}(T_\Omega)$ (resp.\ $A^{p,q}_{\vect s}(T_\Omega)=\widetilde  A^{p,q}_{\vect s}(T_\Omega)$), then $\Ac^{p,q}_{\vect s,0}(D)=\widetilde  \Ac^{p,q}_{\vect s,0}(D)$ (resp.\ $\Ac^{p,q}_{\vect s}(D)=\widetilde  \Ac^{p,q}_{\vect s}(D)$);
		
		\item[\textnormal{(5)}] if $\vect s \succ \frac 1 p (\vect b+\vect d)+\frac{1}{2 p'}\vect m'$ and $\Ac^{p,p}_{\vect s,0}(D)=\widetilde  \Ac^{p,p}_{\vect s,0}(D)$ (resp.\ $\Ac^{p,p}_{\vect s}(D)=\widetilde  \Ac^{p,p}_{\vect s,}(D)$), then $A^{p,p}_{\vect s-\vect b/p,0}(T_\Omega)=\widetilde  A^{p,p}_{\vect s-\vect b/p,0}(T_\Omega)$ (resp.\ $A^{p,p}_{\vect s-\vect b/p}(T_\Omega)=\widetilde  A^{p,p}_{\vect s-\vect b/p}(T_\Omega)$);

		\item[\textnormal{(6)}] if $\vect s \succ \frac{1}{2 q}\vect m+\left(\frac{1}{2 \min(p,p')}-\frac{1}{2 q} \right)_+ \vect m'$, then $\Ac^{p,q}_{\vect s}(D)=\widetilde  \Ac^{p,q}_{\vect s}(D)$  and $\Ac^{p,q}_{\vect s,0}(D)=\widetilde  \Ac^{p,q}_{\vect s,0}(D)$.
	\end{enumerate}
\end{prop}

%Notice that, by (2) and (6), both $\mathring\Bc^{\vect s}_{p,q}(\Nc,\Omega)$ and $\Bc^{\vect s}_{p,q}(\Nc,\Omega)$ are complete.

We observe explicitly that in~\cite[Corollary 5.11]{CalziPeloso} the assumption $\vect s \succ \frac1 p(\vect b+\vect d)+\frac{1}{2 q'}\vect m'$ is redundant (as the assumption $\vect s \succ \frac1 q \vect b +\frac 1 p\vect d+\frac{1}{2 q'}\vect m'$ would be redundant in (6) above), as it is implied by the condition $\vect s\succ \frac{1}{2 q}\vect m+\left(\frac{1}{2 \min(p,p')}-\frac{1}{2 q} \right)_+ \vect m'$. Indeed, $\left(\frac{1}{ \min(p,p')}-\frac{1}{ q} \right)_+ \vect m'\Meg \left(\frac 1 {q'}-\frac1 p\right)\vect m'\Meg \frac 2 p \vect d+\frac{1}{q'}\vect m' $ since $2\vect d \prec -\vect m-\vect m'$. 

We also mention that, if $r=2$ (so that $\Omega$ is isomorphic to either a quadrant or a Lorentz cone), then combining~\cite[Theorems 6.6 and 6.8]{BekolleGonessaNana} (the latter being a consequence of~\cite[Theorem 1.2]{BourgainDemeter}) with~\cite[Theorem 5.10]{CalziPeloso} (cf.~also~\cite[remarks following Remark 2.6]{Paralipomena}), we see that $A^{p,q}_{\vect s}(T_\Omega)=\widetilde  A^{p,q}_{\vect s}(T_\Omega)$  and $A^{p,q}_{\vect s,0}(T_\Omega)=\widetilde  A^{p,q}_{\vect s,0}(T_\Omega)$ if and only if
\[
\vect s \succ \frac{1}{2q}\vect m+\frac 1 2\left(\frac{1}{\min(2,p)} -\frac{1}{ q}   \right)_+\vect m',\frac 1 p \vect d+\frac{1}{2 q'}\vect m' .
\]
By transference, under the same condition we also have $\Ac^{p,q}_{\vect s}(D)=\widetilde  \Ac^{p,q}_{\vect s}(D)$  and $\Ac^{p,q}_{\vect s,0}(D)=\widetilde  \Ac^{p,q}_{\vect s,0}(D)$.

\begin{proof}
	(1) The proof is analogous to that of~\cite[Theorem 4.23, (1)]{CalziPeloso}.
	
	(2) This follows from~\cite[Theorem 4.26]{CalziPeloso} and Remark~\ref{oss:3}.
	
	(3) By~\cite[Proposition 5.4]{CalziPeloso}, there are continuous linear mappings $\Bc\colon \Ac^{q,q}_{\vect s,0}(D)\to \mathring \Bc^{-\vect s}_{q,q}(\Nc,\Omega)$ and $\Bc'\colon A^{p,q}_{\vect s,0}(T_\Omega)\to \mathring B^{-\vect s}_{p,q}(F,\Omega)$ such that $\Ec \Bc =I$ and $\Ec \Bc'=I$. If $u\in \Ac^{p,q}_{\vect s,0}(D)\cap \Ac^{q,q}_{\vect s,0}(D)$, then $u(\zeta,\,\cdot\,)\in A^{p,q}_{\vect s,0}(T_\Omega)$ for almost every $\zeta\in E$, so that Lemma~\ref{lem:2}, Remark~\ref{oss:3},   (1) of Proposition~\ref{prop:11}, and Proposition~\ref{prop:1} imply that  $(\Bc u)(\zeta,\,\cdot\,)=\Bc'(u(\zeta,\,\cdot\,))$ for almost every $\zeta\in E$, and that $\Bc$ induces a continuous linear mapping of $ \Ac^{p,q}_{\vect s,0}(D)$ into $\mathring \Bc^{-\vect s}_{p,q}(\Nc,\Omega)$.  Thus, $\Ac^{p,q}_{\vect s,0}(D)\subseteq \widetilde \Ac^{p,q}_{\vect s,0}(D)$ continuously. The inclusion $\Ac^{p,q}_{\vect s}(D)\subseteq \widetilde \Ac^{p,q}_{\vect s}(D)$ is proved similarly (cf.~the proof of~\cite[Proposition 5.4]{CalziPeloso}).
	The remaining inclusions are consequences of Lemma~\ref{lem:2} and~\cite[Proposition 5.4]{CalziPeloso}.
		
	(4) This follows from Remark~\ref{oss:3} and (1) of Proposition~\ref{prop:11}.
	
	(5) This follows from~\cite[Theorem 6.3]{Paralipomena}.
	
	(6) This follows from (4) and~\cite[Corollary 5.11]{CalziPeloso}.
\end{proof}

\begin{prop}\label{prop:7}
	Take $p,q\in (0,\infty]$, and $\vect s,\vect s'\in \R^r$ such that the following hold:
	\begin{itemize}
		\item $\vect s\succ \frac 1 q\vect b+\frac 1 p \vect d+\frac{1}{2 q'}\vect m'$;
	
		\item $\vect s' \prec \vect b+\vect d-\frac 1 2 \vect m$;
		
		\item $\vect s+\vect s'\prec \frac 1 q \vect b+\frac 1 p\vect d-\frac{1}{2 \max(1,q)}\vect m'$.
	\end{itemize}
	Then, there is a constant $c\neq 0$ such that
	\[
	\langle u\vert u'\rangle= c \int_D (\Ec u)\overline{\Ec u'*I_\Omega^{\vect s'-\vect b-\vect d}} (\Delta_\Omega^{-\vect s'}\circ \rho)\,\dd \nu_D
	\]
	for every $u,u'\in \Sc_{\overline{\Omega'}}(\Nc)$.
	In particular, the sesquilinear form
	\[
	(f,g)\mapsto \int_D f \overline g (\Delta^{-\vect s'}\circ \rho)\,\dd \nu_D
	\]
	induces a continuous sesquilinear form on $\widetilde \Ac^{p,q}_{\vect s}(D)\times \widetilde \Ac^{p',q'}_{\vect b/\min(1,q)+\vect d/\min(1,p)-\vect s-\vect s'}(D)$.
\end{prop}

As a consequence of Corollary~\ref{cor:1}, the above sesquilinear form induces an antilinear isomorphism of $\widetilde\Ac^{p',q'}_{\vect b/\min(1,q)+\vect d/\min(1,p)-\vect s-\vect s'}(D)$ onto $\widetilde \Ac^{p,q}_{\vect s,0}(D)'$.

\begin{proof}
	The proof is analogous to that of~\cite[Proposition 5.12]{CalziPeloso}.\footnote{We mention explicitly that the statement of the cited result contains a typo: instead of $\Delta^{\vect{s''}}_\Omega$, there should be $\Delta^{\vect{s''}-(\vect b+\vect d)}_\Omega$ in the first display formula, as well as in the first line of the formula in display in the proof. }
\end{proof}

\begin{cor}\label{cor:1}
	Take $p,q\in (0,\infty]$ and $\vect s\in \R^r$. Then, the continuous sesquilinear form $\langle \,\cdot\,\vert \,\cdot\,\rangle\colon \Bc^{\vect s}_{p,q}(\Nc,\Omega)\times \Bc^{-\vect s-(1/q-1)_+\vect b-(1/p-1)_+\vect d}_{p',q'}(\Nc,\Omega)\to \C$ induces an antilinear isomorphism of $\Bc^{-\vect s-(1/q-1)_+\vect b-(1/p-1)_+\vect d}_{p',q'}(\Nc,\Omega)$ onto $\mathring \Bc^{\vect s}_{p,q}(\Nc,\Omega)' $. 
\end{cor}

\begin{proof}
	By Proposition~\ref{prop:6}, we may assume that $-\vect s$  is as large as we please, so that the assertion follows from Proposition~\ref{prop:7}, combined with (6) of Proposition~\ref{prop:6}, Theorem~\ref{teo:2}, and Proposition~\ref{prop:9}.
\end{proof}

\subsection{Bergman Projectors}

\begin{prop}
	Take $p,q\in [1,\infty]$, $\vect s\in \R^r$, and $\vect s'\prec \vect b+\vect d-\frac 1 2 \vect m$. If $P_{\vect s'}$ induces an endomorphism of $\Lc^{p,q}_{\vect s,0}(D)$ (resp.\ a continuous linear mapping of $\Lc^{p,q}_{\vect s,0}(D)$ into $\Lc^{p,q}_{\vect s}(D)$), then the following hold:
	\begin{enumerate}
		\item[\textnormal{(1)}] $\vect s\succ \frac{1}{2 q}\vect m $ (resp.\  $\vect s\Meg \vect 0$ if $q=\infty$) and $\vect s \succ\frac 1 q \vect b+\frac 1 p\vect d +\frac{1}{2 q'} \vect{m'}$;
		
		\item[\textnormal{(2)}] $\vect{s'}\prec \frac{1}{\min(p,p')}\vect d-\frac{1}{2 \min(p,p')} \vect{m'}$;
		
		\item[\textnormal{(3)}] $\vect s+\vect{s'}\prec  \vect b+\vect d-\frac{1}{2 q'}\vect m$ or $\vect s+\vect{s'}\meg \vect b+\vect d$ if $q'=\infty$, and $\vect s+\vect s'\prec \frac 1 q \vect b+\frac 1 p\vect d-\frac{1}{2 q} \vect{m'}$;
		
		\item[\textnormal{(4)}] $P_{\vect s'}$ induces continuous linear projectors of $\Lc^{p,q}_{\vect s}(D)$ and   $\Lc^{p',q'}_{\vect b+\vect d-\vect s-\vect s'}(D)$ onto $\Ac^{p,q}_{\vect s}(D)$ and $\Ac^{p',q'}_{\vect b+\vect d-\vect s-\vect s'}(D)$, respectively, such that
		\[
		\int_D f \overline{P_{\vect s'}g} (\Delta^{\vect s'}\circ \rho)\,\dd \nu_D= \int_D (P_{\vect s'} f) \overline{g} (\Delta^{\vect s'}\circ \rho)\,\dd \nu_D
		\]
		for every $f\in \Lc^{p,q}_{\vect s}(D)$ and for every $g\in \Lc^{p',q'}_{\vect b+\vect d-\vect s-\vect s'}(D)$.
	\end{enumerate}
\end{prop}

Notice that (4) uniquely determines $P_{\vect s'}$ on $\Lc^{p,q}_{\vect s}(D)$ and $\Lc^{p',q'}_{\vect b+\vect d-\vect s-\vect s'}(D)$, since conditions (1)--(3) ensure that $B^{-\vect s'}_{(\zeta,z)}\in \Ac^{p,q}_{\vect s}(D)\cap\Ac^{p',q'}_{\vect b+\vect d-\vect s-\vect s'}(D)$ for every $(\zeta,z)\in D$ (cf.~Proposition~\ref{prop:3}).

\begin{proof}
	The proof is analogous to those of~\cite[Propositions 5.20 and 5.21]{CalziPeloso}.
\end{proof}

\begin{teo}\label{teo:3}
	Take $p,q\in [1,\infty]$ and $\vect s,\vect s'\in \R^r$.
	Assume that the following hold:
	\begin{itemize}
		\item $\vect s \succ \frac 1 q \vect b+\frac 1 p \vect d+\frac{1}{2 q'}\vect m'$ and $\vect s \succ \frac{1}{2 q}\vect m$ (resp.\ $\vect s \Meg \vect 0$ if $q=\infty$);
		
		\item $\vect s'\prec \vect b+\vect d-\frac 1 2 \vect m$;
		
		\item $\vect s+\vect s'\prec \frac 1 q \vect b+\frac 1 p \vect d-\frac{1}{2 q}\vect m'$.
	\end{itemize}
	Then, the following conditions are equivalent:
	\begin{enumerate}
		\item[\textnormal{(1)}] $\Ac^{p,q}_{\vect s,0}(D)=\widetilde \Ac^{p,q}_{\vect s,0}(D)$ (resp.\ $\Ac^{p,q}_{\vect s}(D)=\widetilde \Ac^{p,q}_{\vect s}(D)$) and $\Ac^{p',q'}_{\vect b+\vect d-\vect s-\vect s'}(D)=\widetilde\Ac^{p',q'}_{\vect b+\vect d-\vect s-\vect s'}(D)$;
		
		\item[\textnormal{(2)}] $P_{\vect s'}$ induces a continuous linear mapping of $\Lc^{p,q}_{\vect s,0}$ into $\Lc^{p,q}_{\vect s}(D)$ and $\vect s \succ \vect 0$ (resp.\ $\vect s \Meg \vect 0$);
		
		\item[\textnormal{(3)}] $P_{\vect s'}$ induces a continuous linear projector of $\Lc^{p,q}_{\vect s,0}$ onto $\Ac^{p,q}_{\vect s,0}(D)$ (resp.\ of $\Lc^{p,q}_{\vect s}$ onto $\Ac^{p,q}_{\vect s}(D)$) and of $\Lc^{p',q'}_{\vect b+\vect d-\vect s-\vect s'}(D)$ onto $\Ac^{p',q'}_{\vect b+\vect d-\vect s-\vect s'}(D)$;
		
		\item[\textnormal{(4)}] $\vect s \succ \frac{1}{2 q}\vect m$ (resp.\ $\vect s \Meg \vect 0$ if $q=\infty$) and the sesquilinear mapping
		\[
		(f,g)\mapsto \int_D f \overline g (\Delta^{-\vect s'}_\Omega\circ \rho)\,\dd \nu_D
		\]
		induces an antilinear isomorphism of $\Ac^{p',q'}_{\vect b+\vect d-\vect s-\vect s'}(D)$ onto $\Ac^{p,q}_{\vect s,0}(D)'$ (resp.\ onto the dual of the closed vector subspace of $\Ac^{p,q}_{\vect s}(D)$ generated by the $B^{\vect s'}_{(\zeta,z)}$, $(\zeta,z)\in D$);
		
		\item[\textnormal{(5)}] properties $(\Lc')^{p,q}_{\vect s,\vect s',0}$ (resp.\  $(\Lc')^{p,q}_{\vect s,\vect s'}$) and $(\Lc')^{p',q'}_{\vect b+\vect d-\vect s-\vect s',\vect s'}$ hold;
		
		\item[\textnormal{(6)}] property $(\Lc)^{p,q}_{\vect s,\vect s',0}$ (resp.\  $(\Lc)^{p,q}_{\vect s,\vect s'}$) holds.
	\end{enumerate} 
\end{teo}

The proof is analogous to that of~\cite[Corollary 4.7]{Paralipomena}.

In particular, we have the following transference result.

\begin{cor}\label{cor:2}
	Take $p,q\in[1,\infty]$ and $\vect s,\vect s'\in \R^r$. Then, the following hold:
	\begin{itemize}
		\item if $\vect s'\prec \vect d-\frac 1 2 \vect m$ and $P_{\vect s'}$ induces a continuous linear projector of $\Lc^{p,q}_{\vect s,0}(T_\Omega)$ onto $\Ac^{p,q}_{\vect s,0}(T_\Omega)$ (resp.\ of $\Lc^{p,q}_{\vect s}(T_\Omega)$ onto $\Ac^{p,q}_{\vect s}(T_\Omega)$), then $P_{\vect s'+\vect b}$ induces a continuous linear projector of $\Lc^{p,q}_{\vect s,0}(D)$ onto $\Ac^{p,q}_{\vect s,0}(D)$ (resp.\ of $\Lc^{p,q}_{\vect s}(D)$ onto $\Ac^{p,q}_{\vect s}(D)$);
		
		\item if $\vect s'\prec \vect b+\vect d-\frac 1 2 \vect m$ and $P_{\vect s'}$ induces a continuous linear projector of $\Lc^{p,p}_{\vect s,0}(D)$ onto $\Ac^{p,p}_{\vect s,0}(D)$ (resp.\ of $\Lc^{p,p}_{\vect s}(D)$ onto $\Ac^{p,p}_{\vect s}(D)$), then   $P_{\vect s'}$ induces a continuous linear projector of $\Lc^{p,p}_{\vect s-\vect b/p,0}(T_\Omega)$ onto $\Ac^{p,p}_{\vect s-\vect b/p,0}(T_\Omega)$ (resp.\ of $\Lc^{p,p}_{\vect s-\vect b/p}(T_\Omega)$ onto $\Ac^{p,p}_{\vect s-\vect b/p}(T_\Omega)$).
	\end{itemize}
\end{cor}

\smallskip

\section*{Concluding remarks}
On the one hand, we have compared two parallel theories of mixed norm Bergman spaces on
homogeneous Siegel domains. On the other hand, we have extended part of the
theory for the spaces $A^{p,q}_{\vect s}$ to the spaces
$\Af^{p,q}_{\vect s}$.  In doing this, we hope we have shed some light
on the technically demanding subject of function theory on such
domains.  We believe that this is a lively area of research, that in
recent times has drawn the interest of many scholars. We mention that
the \v Silov boundary of $D$ naturally appear in the extension of
the Paley--Wiener and Bernstein spaces of entire functions to higher
complex dimensions, see in particular~\cite{PWS,Bernstein,CarlesonBernstein}.

\end{document}